\documentclass[11pt]{amsart}
\usepackage[margin=1in]{geometry}

\usepackage{amssymb}
\usepackage{amsthm}
\usepackage{amsmath}
\usepackage{mathrsfs}
\usepackage{amsbsy}
\usepackage{bm}
\usepackage{hyperref}
\usepackage{tikz}
\usepackage{array}
\usepackage{enumerate}
\usepackage{enumitem}
\usepackage{bbm}
\usepackage{comment}
\usepackage{mathtools}
\usepackage{tabu}
\usepackage{makecell} 
\usepackage{colortbl}
\usepackage{xcolor}

\DeclareFontFamily{U}{mathx}{}
\DeclareFontShape{U}{mathx}{m}{n}{<-> mathx10}{}
\DeclareSymbolFont{mathx}{U}{mathx}{m}{n}
\DeclareMathAccent{\widecheck}{0}{mathx}{"71}

\definecolor{LightBlue}{rgb}{0,0.8,1} 
\definecolor{Green}{rgb}{0,0.863,0}
\definecolor{DarkGreen}{rgb}{0,0.5,0}
\definecolor{MildGreen}{RGB}{0,200,0}
\definecolor{MildRed}{RGB}{200,0,0}
\definecolor{Turquoise}{RGB}{0,225,190}
\definecolor{Periwinkle}{RGB}{190,190,255}
\definecolor{NormalGreen}{rgb}{0,0.8,0}
\definecolor{LightGreen}{rgb}{0,0.922,0}
\definecolor{Magenta}{rgb}{0.784,0,0.784}
\definecolor{Yellow}{rgb}{0.95,0.95,0}
\definecolor{lavender}{rgb}{0.4,0,1}
\definecolor{peach}{rgb}{1,0.43,0.39} 
\definecolor{DarkPink}{rgb}{1,0,0.45} 
\definecolor{NewBlue}{RGB}{0,200,255} 
\definecolor{Teal}{rgb}{0,0.784,0.784}
\definecolor{Gold}{rgb}{0.929,0.784,0.392}
\definecolor{Purple}{RGB}{180,0,255}
\definecolor{Orange}{RGB}{255,180,100}

\hypersetup{colorlinks=true, citecolor=LightBlue, linkcolor=lavender,urlcolor=lavender}

\usepackage{hhline}
\allowdisplaybreaks
\usepackage[noadjust]{cite}

\usepackage{caption}
\usepackage[noabbrev,capitalise,nameinlink]{cleveref}
\crefname{conjecture}{Conjecture}{Conjectures}

\newtheorem{theorem}{Theorem}[section]
\newtheorem{proposition}[theorem]{Proposition}
\newtheorem{corollary}[theorem]{Corollary}

\newtheorem{lemma}[theorem]{Lemma}

\theoremstyle{definition}
\newtheorem{definition}[theorem]{Definition}
\newtheorem{remark}[theorem]{Remark}
\newtheorem{example}[theorem]{Example} 

\usepackage{etoolbox}

\newcommand{\includeSymbol}[1]{\ensuremath{%
	\mathchoice
		{\raisebox{-.4mm}{\includegraphics[height=2.1ex]{#1}}}	
		{\raisebox{-.4mm}{\includegraphics[height=2.1ex]{#1}}}
		{\raisebox{-.3mm}{\includegraphics[height=1.6ex]{#1}}}
		{\raisebox{-.2mm}{\includegraphics[height=1ex]{#1}}}
}} 

\robustify{\includeSymbol}
\newcommand{\ULT}{\includeSymbol{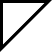}\!} 

\robustify{\includeSymbol}
\newcommand{\LRT}{\includeSymbol{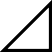}} 

\robustify{\includeSymbol}
\newcommand{\Bump}{\includeSymbol{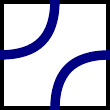}}

\robustify{\includeSymbol}
\newcommand{\Cross}{\includeSymbol{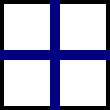}}

\robustify{\includeSymbol}
\newcommand{\HalfBump}{\includeSymbol{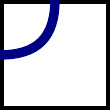}}

\robustify{\includeSymbol}
\newcommand{\Tri}{\includeSymbol{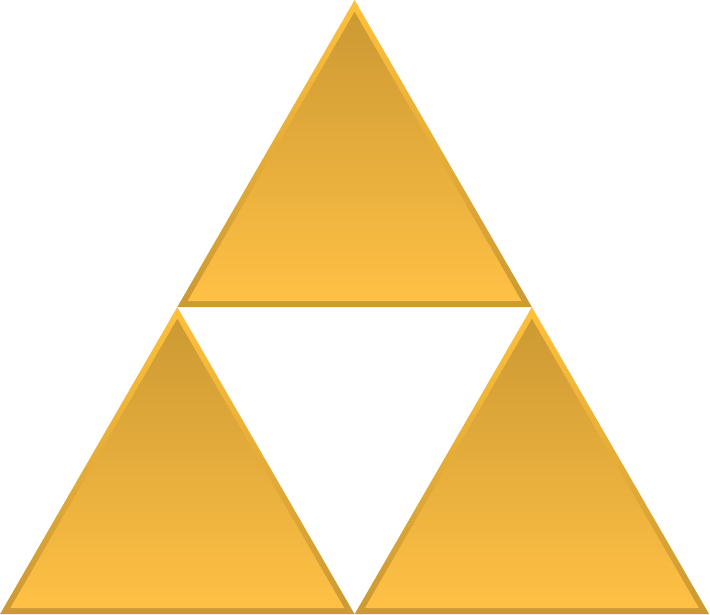}}

\newcommand{\SSS}{\mathfrak{S}} 
\newcommand{\PD}{\mathrm{PD}} 
\newcommand{\LT}{\mathrm{LT}} \newcommand{\ID}{\mathrm{ID}} 
\newcommand{\leqC}{\leq_{\mathrm{chute}}}
\newcommand{\leqL}{\leq} 
\newcommand{\Le}{\reflectbox{$\mathrm{L}$}} 
\newcommand{\T}{\mathrm{T}}
\newcommand{\IT}{\mathrm{IT}}
\newcommand{\row}{\mathsf{row}}
\newcommand{\col}{\mathsf{col}}
\newcommand{\CIT}{\mathrm{CIT}} 
\newcommand{\C}{\mathsf{chute}} 
\newcommand{\hook}{\mathsf{hook}} 
\newcommand{\dd}{\ast} 
\newcommand{\xx}{x_0}
\newcommand{\yy}{y_0} 
\newcommand{\BB}{\mathsf{B}} 
\newcommand{\NW}{\mathsf{NW}}
\newcommand{\NE}{\mathsf{NE}}
\newcommand{\SW}{\mathsf{SW}}
\newcommand{\SE}{\mathsf{SE}}
 
\newcommand{\RR}{\mathbf{R}} 
\newcommand{\RRR}{\mathfrak{R}}
\newcommand{\QQQ}{\mathfrak{Q}}

\definecolor{goodgreen}{rgb}{0.01, 0.75, 0.24}

\newcommand{\dfn}[1]{\textcolor{blue}{\emph{#1}}}

\begin{document}

\title[]{Chute Move Posets are Lattices} 
\subjclass[2010]{}

\author[]{Ilani Axelrod-Freed}
\address[]{Department of Mathematics, Massachusetts Institute of Technology, Cambridge, MA 02139, USA}
\email{ilani\_af@mit.edu} 

\author[]{Colin Defant}
\address[]{Department of Mathematics, Harvard University, Cambridge, MA 02138, USA}
\email{colindefant@gmail.com} 

\author[]{Hanna Mularczyk} 
\address[]{Department of Mathematics, Massachusetts Institute of Technology, Cambridge, MA 02139, USA}
\email{hannamul@mit.edu} 

\author[]{Son Nguyen} 
\address[]{Department of Mathematics, Massachusetts Institute of Technology, Cambridge, MA 02139, USA}
\email{sonnvt@mit.edu} 

\author[]{Katherine Tung} 
\address[]{Department of Mathematics, Harvard University, Cambridge, MA 02138, USA}
\email{katherinetung@college.harvard.edu}

\begin{abstract}
For each permutation $w$, we consider the set $\mathrm{PD}(w)$ of reduced pipe dreams for $w$, partially ordered so that cover relations correspond to (generalized) chute moves. Settling a conjecture of Rubey from 2012, we prove that $\mathrm{PD}(w)$ is a lattice. To establish this result, we provide a global description of the partial order on $\mathrm{PD}(w)$ by showing that $\mathrm{PD}(w)$ is isomorphic to a poset consisting of objects called \emph{Lehmer tableaux}. In addition, we prove that $\mathrm{PD}(w)$ is a semidistributive polygonal lattice whose polygons are all diamonds or pentagons. 
\end{abstract} 

\maketitle

\section{Introduction}\label{sec:intro} 
Consider the $n$-th staircase Young diagram $\ULT_n$, which is the Young diagram of the partition $(n,n-1,n-2,\ldots,1)$ drawn using English conventions. We index the rows of $\ULT_n$ as $1,\ldots,n$ from top to bottom. A \dfn{pipe dream} is a filling of the boxes of $\ULT_n$ in which each box along the southeast boundary is filled with an \dfn{elbow tile} $\HalfBump$ and every other box is filled with either a \dfn{cross tile} $\Cross$ or a \dfn{bump tile} $\Bump$. This produces an arrangement of $n$ \dfn{pipes} that travel from the west side of $\ULT_n$ to the north side of $\ULT_n$ (see \cref{fig:chute_move}). We label each pipe with an element of the set $[n]=\{1,\ldots,n\}$ according to the row where it starts on the west side of $\ULT_n$. Let $\SSS_n$ denote the symmetric group of permutations of $[n]$. Given a pipe dream $P$, we can read the labels of the pipes along the north side of $\ULT_n$ from left to right to obtain a permutation $w_P\in\SSS_n$ in one-line notation; we say $P$ is a pipe dream for $w_P$. A pipe dream is \dfn{reduced} if no two pipes cross each other more than once. For $w\in\SSS_n$, let $\PD(w)$ denote the set of reduced pipe dreams for $w$. Reduced pipe dreams are fundamental objects in Schubert calculus because they provide combinatorial interpretations of Schubert polynomials \cite{BergeronBilley,FominKirillov,FK,
HPW,KnutsonMiller1,KnutsonMiller2}; they are essentially equivalent to \emph{rc-graphs}, which were introduced by Billey--Jockusch--Stanley \cite{BJS} following related works of Fomin--Kirillov \cite{FominKirillov,FK} and Fomin--Stanley \cite{FS}.  

\begin{figure}[ht]
  \begin{center}
  \includegraphics[height=3.874cm]{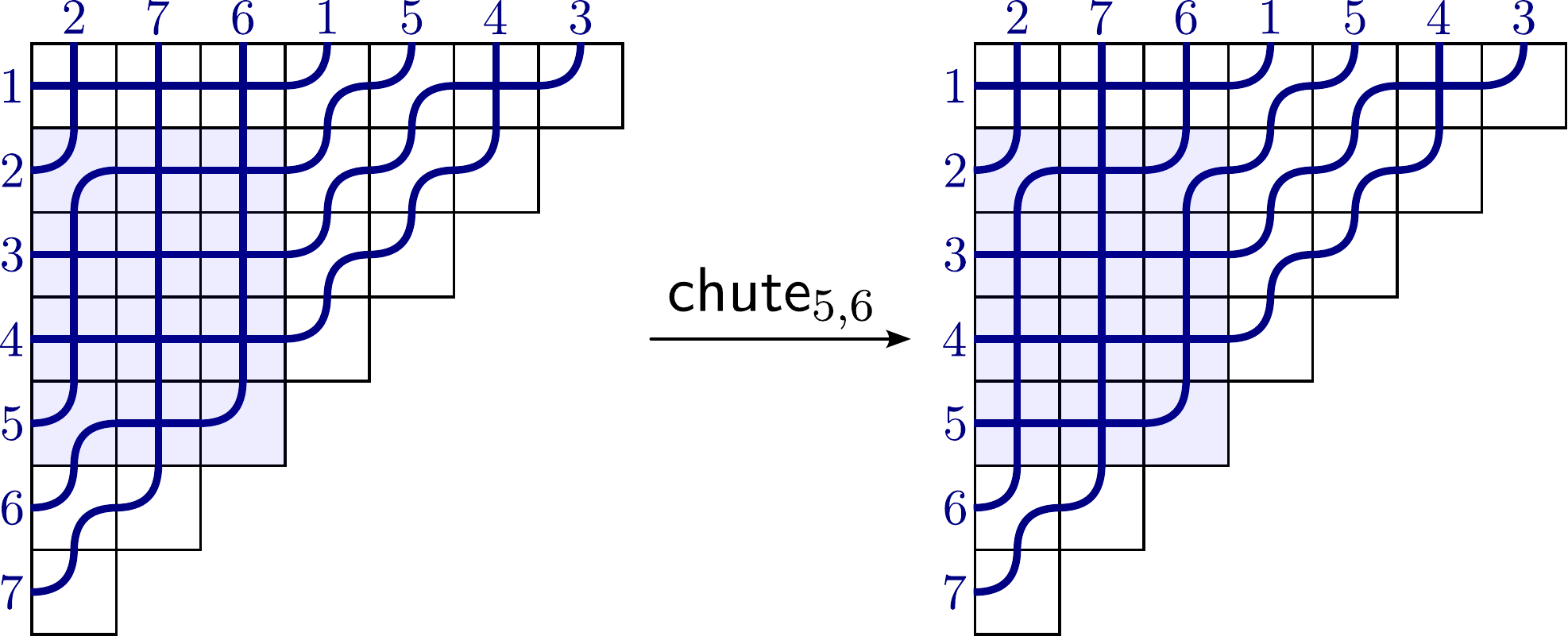}
  \end{center}
\caption{Two reduced pipe dreams in $\PD(2761543)$. The pipe dream on the right is obtained from the one on the left by applying a chute move that takes place in the shaded rectangle. }\label{fig:chute_move} 
\end{figure}

Consider a rectangle $R$ of boxes in $\ULT_n$. We denote by $\NW(R)$, $\NE(R)$, $\SW(R)$, and $\SE(R)$ the boxes in the northwest, northeast, southwest, and southeast corners of $R$, respectively. Now consider a pipe dream $P\in\PD(w)$ such that $\NW(R)$ and $\SW(R)$ are filled with bump tiles, $\SE(R)$ is filled with a bump or elbow tile, and all other boxes in $R$ are all filled with cross tiles. We can apply a \dfn{chute move}\footnote{Chute moves were called \emph{generalized chute moves} in \cite{CPS,Rubey}.} to $P$ by changing the bump tile in $\SW(R)$ to a cross tile and changing the cross tile in $\NE(R)$ to a bump tile (see \cref{fig:chute_move}); this results in a new pipe dream $P'$, which is also in $\PD(w)$. If the cross tile in $P$ in the box $\NE(R)$ involves pipes $i$ and $j$, then we write \[P'=\C_{i,j}(P).\] We say this chute move \dfn{takes place} in the rectangle $R$, and we write $R=\RR(P,P')$. 

\begin{figure}[htbp]
  \begin{center}
  \includegraphics[width=0.78\linewidth]{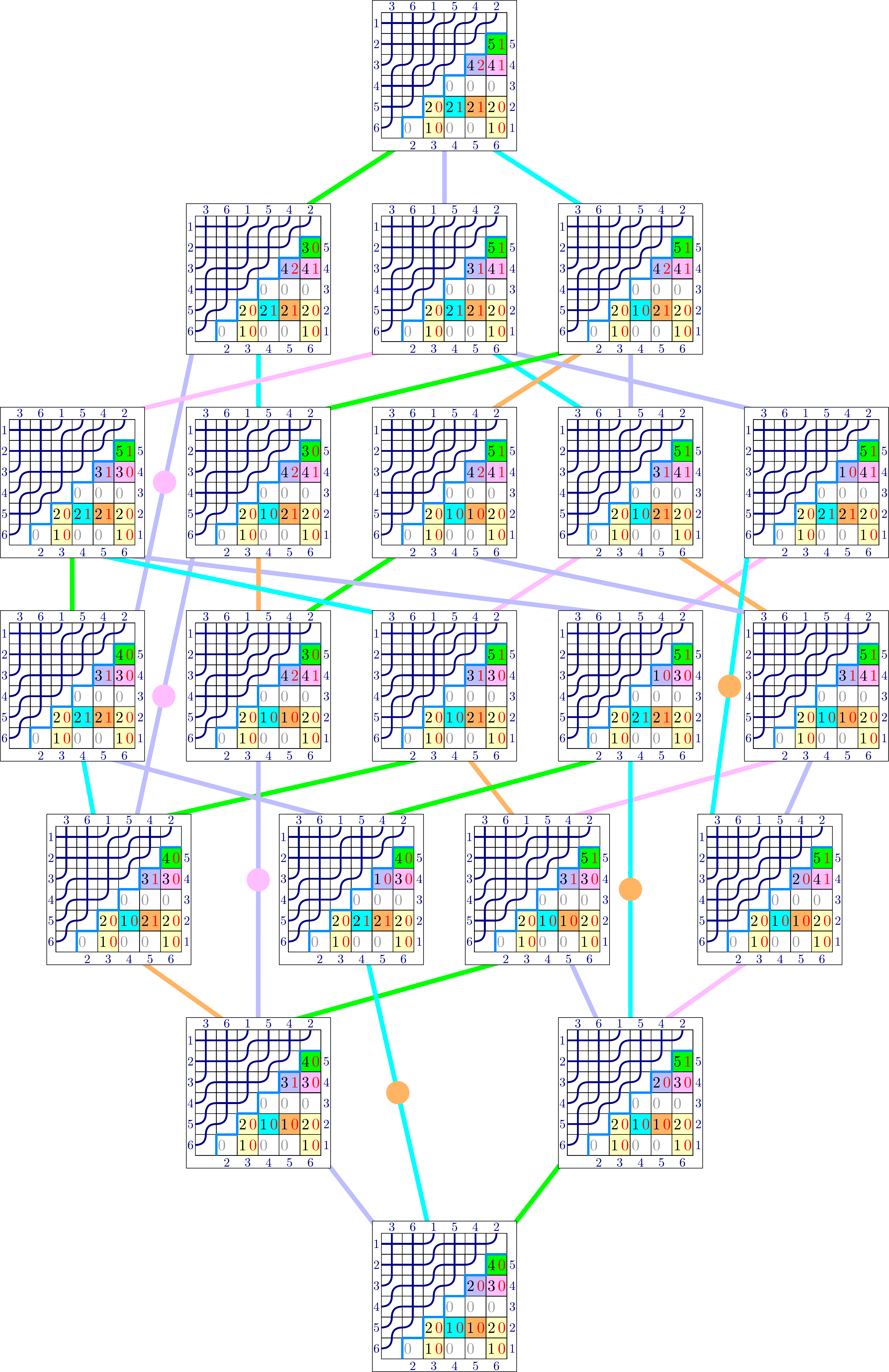}
  \end{center}
\caption{The chute move poset $\PD(361542)$, with each element represented by a pipe dream $P$, the inversions tableau $\Theta(P)$ (with black numbers), and the Lehmer tableau $\Phi(P)$ (with {\color{red}red} numbers). Each edge is colored according to which box corresponds to the pair of pipes involved in the chute move. Discs on edges are colored according to additional boxes that change in the Lehmer tableau. 
}\label{fig:big_lattice} 
\end{figure}

In 2012, building off of work of Bergeron--Billey \cite{BergeronBilley} and Serrano--Stump \cite{SerranoStump}, Rubey \cite{Rubey} defined a partial order $\leqC$ on $\PD(w)$ by declaring that $P\leqC P'$ if one can obtain $P'$ from $P$ by applying a sequence of chute moves. We call the poset $\PD(w)$ equipped with the partial order $\leqC$ the \dfn{chute move poset} of $w$ (see \cref{fig:big_lattice}). Rubey proved that every chute move poset has a minimum element and a maximum element \cite[Theorem~3.8]{Rubey}; his argument was a mild modification of a previous argument due to Bergeron and Billey \cite{BergeronBilley} for a related poset defined using only chute moves that take place in rectangles of height $2$. Rubey also conjectured that $\PD(w)$ is always a lattice \cite[Conjecture~2.8]{Rubey}. In the special case when $w$ begins with $1$ and avoids the pattern $1243$, Ceballos, Padrol, and Sarmiento \cite[Theorem~5.9]{CPS} proved Rubey's conjecture by showing that $\PD(w)$ is isomorphic to a certain $\nu$-Tamari lattice (in fact, every $\nu$-Tamari lattice arises in this way). Our main goal in this article is to fully resolve Rubey's conjecture.\footnote{While writing this article, we were informed that Billey, McCausland, and Minnerath have simultaneously and independently proven Rubey's conjecture using a completely different approach \cite{BMM}. } 

\begin{theorem}\label{thm:main} 
For every permutation $w\in\SSS_n$, the chute move poset $\PD(w)$ is a lattice. 
\end{theorem} 

The primary difficulty in proving \cref{thm:main} comes from the fact that the partial order $\leqC$ is defined locally in terms of covering relations. In order to prove the lattice property, we seek a global description of this partial order. We provide such a description in terms of \emph{inversions tableaux} and \emph{Lehmer tableaux} (see \cref{sec:tableaux} for the definitions). Roughly speaking, inversions tableaux are certain tableaux in bijection with reduced pipe dreams; they were introduced recently by Axelrod-Freed in order to give a new combinatorial formula for Schubert polynomials \cite{AF}. Axelrod-Freed also introduced \emph{Lehmer tableaux}, which are in bijection with inversions tableaux. Thus, letting $\LT(w)$ denote the set of Lehmer tableaux corresponding to reduced pipe dreams in $\PD(w)$, we have a bijection ${\Phi\colon\PD(w)\to\LT(w)}$. There is a natural partial order $\leqL$ on $\LT(w)$ defined globally by componentwise comparison of entries. We call $\LT(w)$ equipped with this partial order a \dfn{Lehmer poset}. We will prove that the map $\Phi$ is a poset isomorphism (\cref{thm:isomorphism}). We will then combine the local and global descriptions of the partial order to complete the proof of \cref{thm:main}. Given two Lehmer tableaux $L,L'\in\LT(w)$, one can easily determine if $L\leq L'$. Therefore, given $P,P'\in\PD(w)$, one can use the isomorphism $\Phi$ to efficiently test whether $P\leqC P'$. 

Our second main theorem shows that 
chute move posets have additional structure beyond the lattice property.  

\begin{theorem}\label{thm:semidistributive} 
For every permutation $w\in\SSS_n$, the lattice $\PD(w)$ is semidistributive and polygonal. Moreover, every polygon in $\PD(w)$ is a diamond or a pentagon. 
\end{theorem} 

Semidistributive lattices (defined in \cref{sec:background}) form an important generalization of distributive lattices. There is a representation theorem for semidistributive lattices due to Reading, Speyer, and Thomas \cite{RST} that is analogous to Birkhoff's representation theorem for distributive lattices \cite{Birkhoff}. Each semidistributive lattice also has a canonical join complex and comes equipped with a bijective \emph{rowmotion} operator \cite{Barnard,Reading,ThomasWilliams}. In addition, the order complex of a semidistributive lattice is either contractible or homotopy equivalent to a sphere \cite{McConville}. \cref{thm:semidistributive} tells us that chute move posets enjoy all of these properties. 

It would be interesting to gain an even deeper understanding of the lattice-theoretic properties of chute move posets. For example, one could attempt to describe the join-irreducible elements and the canonical join complex of $\PD(w)$ (as in \cite{Barnard}).

\subsection{Outline} 
\begin{itemize}
\item \cref{sec:background} establishes useful background information on posets and lattices. 
\item \cref{sec:tableaux} defines inversions and Lehmer tableaux and discusses their basic properties. 
\item \cref{sec:lemmas} establishes some technical lemmas about inversions and Lehmer tableaux.
\item In \cref{sec:transformations}, we discuss two useful transformations that one can apply to a pipe dream, and we describe how these transformations interact with the partial orders under consideration. The first transformation is the transpose, which changes a reduced pipe dream for $w$ into a reduced pipe dream for $w^{-1}$. The second transformation is a new map that we call the \emph{triforce embedding}; it changes a reduced pipe dream for a permutation $w\in\SSS_n$ into a reduced pipe dream for a larger permutation $w^{\Tri}\in\SSS_{2n}$. 
\item  \cref{sec:isomorphism} proves that the bijection $\Phi\colon\PD(w)\to\LT(w)$ is an isomorphism. 
\item In \cref{sec:lattice}, we use the isomorphism $\Phi$, along with a useful result about the triforce embedding from \cref{sec:transformations}, to prove \cref{thm:main}. As a corollary of the proof, we show that chute move posets are polygonal lattices whose polygons are all diamonds or pentagons. 
\item In \cref{sec:semidistributivity}, we complete the proof of \cref{thm:semidistributive} by showing that chute move posets are semidistributive. This proof also makes heavy use of the triforce embedding. 
\end{itemize}

\section{Preliminaries}\label{sec:background} 
Let $Q$ be a finite poset with a partial order $\leq$. For $\alpha,\beta\in Q$ with $\alpha\leq \beta$, the \dfn{interval} between $\alpha$ and $\beta$ is the set $[\alpha,\beta]=\{\gamma\in Q:\alpha\leq \gamma\leq \beta\}$. If $\alpha<\beta$ and $[\alpha,\beta]=\{\alpha,\beta\}$, then we say $\beta$ \dfn{covers} $\alpha$ and write $\alpha\lessdot \beta$. A \dfn{maximum} (respectively, \dfn{minimum}) element of $Q$ is an element $\gamma^*\in Q$ such that $\alpha\leq\gamma^*$ (respectively, $\gamma^*\leq\alpha$) for all $\alpha\in Q$.

A \dfn{polygon} in a poset is an interval $[\alpha,\beta]$ of cardinality at least $4$ that is the union of two chains whose intersection is $\{\alpha,\beta\}$. A \dfn{diamond} is a polygon of cardinality $4$. A \dfn{pentagon} is a polygon of cardinality $5$. 

An \dfn{anti-isomorphism} from a poset $Q$ to a poset $Q'$ is a bijection $\varphi\colon Q\to Q'$ such that for all $\alpha,\beta\in Q$, we have $\alpha\leq\beta$ in $Q$ if and only if $\varphi(\beta)\leq\varphi(\alpha)$ in $Q'$. 

A \dfn{lattice} is a poset $Q$ such that any two elements $\alpha,\beta\in Q$ have a greatest lower bound $\alpha\wedge \beta$ and a least upper bound $\alpha\vee \beta$. We call $\alpha\wedge \beta$ and $\alpha\vee \beta$ the \dfn{meet} and \dfn{join} (respectively) of $\alpha$ and $\beta$. A lattice $Q$ is \dfn{polygonal} if for all $\gamma,\gamma'\in Q$ such that $\gamma$ and $\gamma'$ cover $\gamma\wedge \gamma'$ or are covered by $\gamma\vee \gamma'$, the interval $[\gamma\wedge \gamma',\gamma\vee \gamma']$ is a polygon. 

Let $Q$ be a finite lattice. We say $Q$ is \dfn{meet-semidistributive} if for all $\alpha,\beta\in Q$ such that $\alpha\leq\beta$, the set $\{\gamma\in Q:\gamma\wedge\beta=\alpha\}$ has a maximum element. We say $Q$ is \dfn{join-semidistributive} if for all $\alpha,\beta\in Q$ such that $\alpha\leq\beta$, the set $\{\gamma\in Q:\gamma\vee\alpha=\beta\}$ has a minimum element. We say $Q$ is \dfn{semidistributive} if it is both meet-semidistributive and join-semidistributive. 

\begin{proposition}[{\cite[Theorem~2.56]{Free}}]\label{prop:Free}
Let $Q$ be a finite lattice. Then $Q$ is meet-semidistributive if and only if for every cover relation $\alpha\lessdot\beta$ in $Q$, the set $\{\gamma\in Q:\gamma\wedge\beta=\alpha\}$ has a maximum element. 
\end{proposition} 

A subset $X$ of a poset $Q$ is called \dfn{order-convex} if for all $\alpha,\beta\in X$ such that $\alpha\leq\beta$, the interval $[\alpha,\beta]$ is contained in $X$. 
The next proposition will be useful when applying \cref{prop:Free}. 

\begin{proposition}\label{prop:maximum}
Let $Q$ be a finite lattice. Let $X$ be an order-convex subset of $Q$ with a minimum element. Suppose that for all $\gamma,\gamma'\in X$ satisfying $\gamma\wedge\gamma'\lessdot\gamma$ and $\gamma\wedge\gamma'\lessdot\gamma'$, we have $\gamma\vee\gamma'\in X$. Then $X$ has a maximum element. 
\end{proposition} 
\begin{proof}
Let $\max X$ denote the set of maximal elements of $X$, and suppose instead that $\lvert\max X\rvert\geq 2$. Let $A=\{\alpha\wedge\alpha':\alpha,\alpha'\in\max X,\,\alpha\neq\alpha'\}$. Choose distinct $\beta,\beta'\in\max X$ such that $\beta\wedge\beta'$ is a maximal element of $A$. Because $X$ is order-convex and has a minimum element, we know that $\beta\wedge\beta'\in X$. The elements $\beta$ and $\beta'$ are incomparable in $Q$, so we can find $\gamma,\gamma'\in Q$ such that $\beta\wedge\beta'\lessdot\gamma\leq\beta$ and $\beta\wedge\beta'\lessdot\gamma'\leq\beta'$. Because $X$ is order-convex, we have $\gamma,\gamma'\in X$. Because $\gamma$ and $\gamma'$ both cover $\beta\wedge\beta'$, we must have $\beta\wedge\beta'=\gamma\wedge\gamma'$, so $\gamma\vee\gamma'\in X$ by the hypothesis of the proposition. Now choose $\beta''\in\max X$ such that $\gamma\vee\gamma'\leq\beta''$. We have $\beta\wedge\beta'\lessdot\gamma'\leq\beta''\wedge\beta'$, so $\beta\neq\beta''$. This implies that $\beta''\wedge\beta\in A$. We also have $\beta\wedge\beta'\lessdot \gamma\leq\beta''\wedge\beta$, which contradicts the fact that $\beta\wedge\beta'$ is a maximal element of $A$. 
\end{proof}

\section{Inversions Tableaux and Lehmer Tableaux}\label{sec:tableaux}  

Let $\LRT_{n-1}$ denote the $(n-1)$-th staircase Young diagram drawn by reflecting French conventions across a vertical axis. We index the rows of $\LRT_{n-1}$ as $1,\ldots,n-1$ from bottom to top, and we index the columns as $2,\ldots,n$ from left to right. For $1\leq i<j\leq n$, we index the box in row $i$ and column $j$ by the pair $(i,j)$. By identifying $\LRT_{n-1}$ with its set of boxes, we can write \[\LRT_{n-1}=\{(i,j)\in\mathbb Z^2:1\leq i<j\leq n\}.\] 
See \cref{fig:inversions_diagram_1}. 

For $1\leq i<j<k\leq n$, let $\Le(i,j,k)$ be the set $\{(i,j),(i,k),(j,k)\}$ of boxes in $\LRT_{n-1}$. We can visualize $\Le(i,j,k)$ by drawing a rectangle whose southwest, southeast, and northeast corners lie in the boxes $(i,j)$, $(i,k)$, and $(j,k)$, respectively; the northwest corner of this rectangle lies just above the diagonal of $\LRT_{n-1}$. We call $\Le(i,j,k)$ a \dfn{$\Le$-shape}. We define the \dfn{hook} of the box $(i,j)\in\LRT_{n-1}$ to be the set of boxes 
\[\hook(i,j)=\{(i',j):i<i'<j\}\cup\{(i,j'):i<j'<j\}\cup\{(i,j)\}.\]

Let $w\in\SSS_n$. An \dfn{inversion} of $w$ is a pair $(i,j)$ such that $1\leq i<j\leq n$ and ${w^{-1}(i)>w^{-1}(j)}$. The \dfn{inversions diagram} of $w$ is the set $\ID(w)$ of boxes in $\LRT_{n-1}$ indexed by the inversions of $w$; we often represent $\ID(w)$ by coloring its boxes in $\LRT_{n-1}$ (see \cref{fig:inversions_diagram_1}).\footnote{Our conventions for writing permutations in this article are slightly different from those in \cite{AF}; one can translate from one to the other by taking inverses. For example, what we call reduced pipe dreams, inversions diagrams, inversions tableaux, and Lehmer tableaux for a permutation $w$ are what the article \cite{AF} would call reduced pipe dreams, inversions diagrams, inversions tableaux, and Lehmer tableaux for $w^{-1}$.} 

\begin{figure}[ht]
  \begin{center}
  \includegraphics[height=2.857cm]{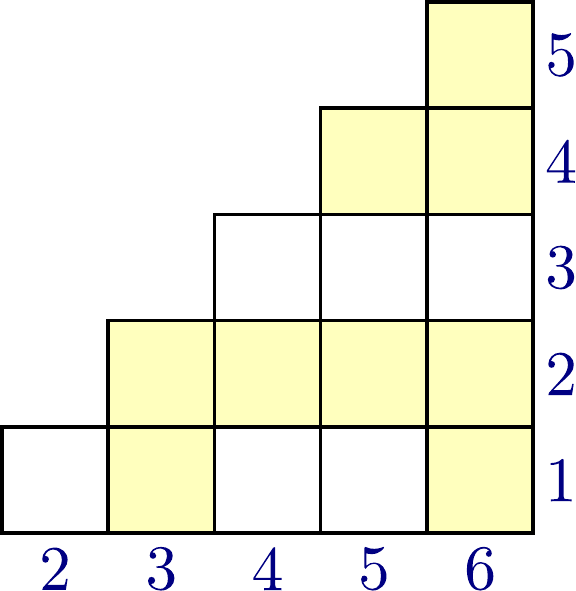}\qquad\qquad\includegraphics[height=2.857cm]{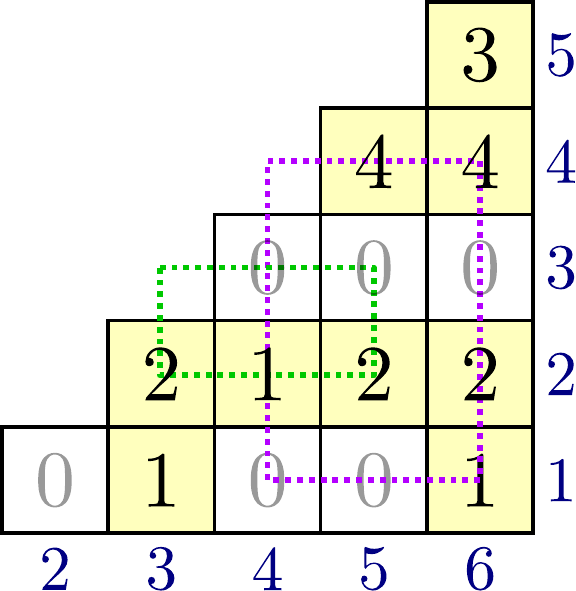}\qquad\qquad\includegraphics[height=2.858cm]{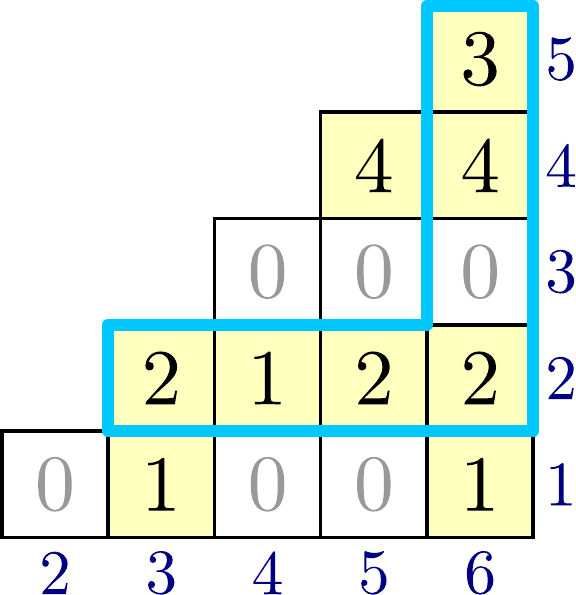}
  \end{center}
\caption{On the left is the inversions diagram of the permutation $w=361542$. In the middle is an inversions tableau for $w$. The {\color{Purple}purple} and {\color{MildGreen}green} dotted rectangles indicate the $\Le$-shapes ${\color{Purple}\Le(1,4,6)}$ and ${\color{MildGreen}\Le(2,3,5)}$. These $\Le$-shapes are balanced because ${\color{Purple}0\leq 1\leq 4}$ and ${\color{MildGreen}0\leq 2\leq 2}$. On the right is the same inversions tableau with ${\color{NewBlue}\hook(2,6)}$ outlined in {\color{NewBlue}blue}. The entries in this hook are $0,1,2,2,2,3,4$; the median of these numbers is $2$, which appears in the southeast corner of the hook, so the hook is balanced.}
\label{fig:inversions_diagram_1} 
\end{figure} 

For $\mathcal{D}\subseteq\LRT_{n-1}$, we define a \dfn{tableau} of shape $\mathcal D$ to be a map $\tau\colon\mathcal D\to\mathbb Z_{\geq 0}$; we represent $\tau$ by filling each box $(i,j)\in\mathcal D$ with the integer $\tau(i,j)$. Let $\T_{\mathcal D}$ denote the set of tableau of shape $\mathcal D$. We view $\T_{\mathcal D}$ as a poset with the \dfn{componentwise comparison} partial order $\leq$, which is defined so that $\tau\leq\tau'$ if and only if $\tau(i,j)\leq\tau'(i,j)$ for all $(i,j)\in\mathcal D$. Let $\mathcal D\subseteq\LRT_{n-1}$, and consider two tableau $\tau,\tau'\colon\mathcal D\to\mathbb Z_{\geq 0}$. The \dfn{componentwise maximum} of two tableaux $\tau,\tau'\in\T_{\mathcal D}$ is the tableau $\tau^\star\colon\mathcal D\to\mathbb Z_{\geq 0}$ defined so that $\tau^\star(i,j)=\max\{\tau(i,j),\tau'(i,j)\}$ for every $(i,j)\in\mathcal D$. 

Given a tableau $\tau\in\T_{\mathcal D}$, we say an integer $k$ \dfn{appears below} (respectively, \dfn{appears above}) a box $(i,j)$ in $\tau$ if there exists a positive integer $i'$ such that $i'<i$ (respectively, $i'>i$) and $\tau(i',j)=k$. Similarly, $k$ \dfn{appears to the left of} (respectively, \dfn{appears to the right of}) $(i,j)$ in $\tau$ if there exists a positive integer $j'$ such that $j'<j$ (respectively, $j'>j$) and $\tau(i,j')=k$. 

We will often need to restrict tableaux to smaller staircase shapes. For $w\in\SSS_n$, the set of boxes in $\ID(w)$ that also belong to $\LRT_{j-1}$ is the inversions diagram $\ID(\widehat w)$ of the permutation $\widehat w\in\SSS_j$ obtained from $w$ by deleting the numbers greater than $j$. Given a tableau $\tau\colon\LRT_{n-1}\to\mathbb Z_{\geq 0}$, we denote by $\tau_{\leq j}$ the tableau obtained by restricting $\tau$ to $\LRT_{j-1}$ (i.e., by deleting all columns strictly to the right of column $j$). 

Let $w\in\SSS_n$. A \dfn{column-injective tableau} for $w$ is a tableau $T\colon\LRT_{n-1}\to\mathbb Z_{\geq 0}$ such that the following conditions hold. 
\begin{enumerate}
\item[(1)] A box is filled with $0$ in $T$ if and only if it is not in $\ID(w)$.
\item[(2)] The nonzero entries in each column of $T$ are distinct. 
\end{enumerate} 
Let $\CIT(w)$ denote the set of column-injective tableaux for $w$.

Let $\tau\colon\LRT_{n-1}\to\mathbb Z_{\geq 0}$ be a tableau. We say a $\Le$-shape $\Le(i,j,k)$ in $\LRT_{n-1}$ is \dfn{balanced} in $\tau$ if \[\min\{\tau(i,j), \tau(j,k)\} \leq \tau(i,k) \leq \max\{\tau(i,j), \tau(j,k)\}\] 
(i.e., $\tau(i,k)$ is weakly between $\tau(i,j)$ and $\tau(j,k)$). For example, the set $\Le(1,4,6)$ is balanced in the tableau in the middle of \cref{fig:inversions_diagram_1} because $1$ lies weakly between $0$ and $4$. We say $\tau$ is \dfn{balanced} if every $\Le$-shape in $\LRT_{n-1}$ is balanced in $\tau$. We say $\hook(i,j)$ is \dfn{balanced} in $\tau$ if $\tau(i,j)$ is the median of its entries in $\tau$. For example, $\hook(2,6)$ is balanced in the tableau on the right of \cref{fig:inversions_diagram_1} because $2$ is the median of $0,1,2,2,2,3,4$. It is straightforward to show that a tableau $\tau\colon\LRT_{n-1}\to\mathbb Z_{\geq 0}$ is balanced if and only if $\hook(i,j)$ is balanced in $\tau$ for every $(i,j)\in\LRT_{n-1}$. 

\begin{definition}[{\cite{AF}}] 
Let $w\in\SSS_n$. An \dfn{inversions tableau} for $w$ is a balanced column-injective tableau for $w$ in which the entries in each row $i$ are all at most $i$. Let $\IT(w)$ denote the set of inversions tableaux for $w$. 
\end{definition} 
\cref{fig:inversions_diagram_1} depicts an inversions tableau for the permutation $361542$ in the middle and on the right.  

\begin{definition}
Let $T$ be a column-injective tableau for a permutation $w$. The \dfn{Lehmer form} of $T$ is the tableau $\Lambda(T)\in\T_{\ID(w)}$ defined as follows. For each $(i,j)\in\ID(w)$, let $\Lambda(T)(i,j)$ be the number of integers $k\in[T(i,j)-1]$ that do not appear below box $(i,j)$ in $T$. If $T$ is an inversions tableau for $w$, then we call $\Lambda(T)$ a \dfn{Lehmer tableau} for $w$. Let $\LT(w)=\Lambda(\IT(w))$ denote the set of Lehmer tableaux for $w$. We view $\LT(w)$ as a poset under the componentwise comparison order. 
\end{definition} 

It is straightforward to see that the map $\Lambda\colon\CIT(w)\to\T_{\ID(w)}$ is a bijection. Indeed, one can reconstruct each column of a column-injective tableau $T$ from its Lehmer form by working from bottom to top. Hence, by restricting to the set of inversions tableau for $w$, we obtain a bijection $\Lambda\colon\IT(w)\to\LT(w)$. This map is illustrated in \cref{fig:pipe_inversions_Lehmer}. 

Consider a reduced pipe dream $P\in\PD(w)$. The pipes labeled $i$ and $j$ cross each other if and only if $(i,j)$ is an inversion of $w$. Moreover, if $(i,j)$ is an inversion of $w$, then pipes $i$ and $j$ cross exactly once because $P$ is reduced. For each $(i,j)\in\ID(w)$, let $\row_P(i,j)$ be the index of the row in $P$ where pipes $i$ and $j$ cross (rows in pipe dreams are indexed from top to bottom). Let $\Theta(P)\colon\LRT_{n-1}\to\mathbb Z_{\geq 0}$ be the tableau defined by 
\[\Theta(P)(i,j)=\begin{cases}   \row_P(i,j) & \text{if } (i,j)\in\ID(w); \\
    0 & \text{otherwise}.
\end{cases}\]
\cref{fig:pipe_inversions_Lehmer} illustrates the map $\Theta$. For example, in the pipe dream on the left of that figure, the pipes $5$ and $6$ cross each other in row $3$, so the tableau in the middle has a $3$ in the box $(5,6)$. 

Axelrod-Freed's primary motivation for introducing inversions tableaux stems from the following theorem.

\begin{theorem}[{\cite{AF}}]
For each $w\in\SSS_n$, the map $\Theta$ is a bijection from $\PD(w)$ to $\IT(w)$. 
\end{theorem} 

We now have a bijection $\Theta$ from the set $\PD(w)$ of reduced pipe dreams for $w$ to the set $\IT(w)$ of inversions tableaux for $w$, and we also have the bijection $\Lambda$ from $\IT(w)$ to the set $\LT(w)$ of Lehmer tableaux for $w$. Thus, we have the composite bijection 
\[\Phi=\Lambda\circ\Theta\colon\PD(w)\to\LT(w).\] 
See \cref{fig:pipe_inversions_Lehmer}. See also \cref{fig:big_lattice}, where each element of a chute move poset is represented as a reduced pipe dream, an inversions tableau, and a Lehmer tableau, all stuck together (we use black numbers for entries in the inversions tableau and red numbers for entries in the Lehmer tableau). One could alternatively view the partial order in that figure as the componentwise comparison order on the Lehmer tableau; we will see in \cref{sec:isomorphism} that these two posets are indeed isomorphic (see \cref{thm:isomorphism}). In \cref{fig:big_lattice}, each edge of the Hasse diagram is colored according to the leftmost box that changes in the Lehmer tableau when one traverses the edge. Some edges also have an extra colored disc, corresponding to the color of another box that changes in the Lehmer tableau. 

\begin{figure}[ht]
  \begin{center}
  \includegraphics[height=3.556cm]{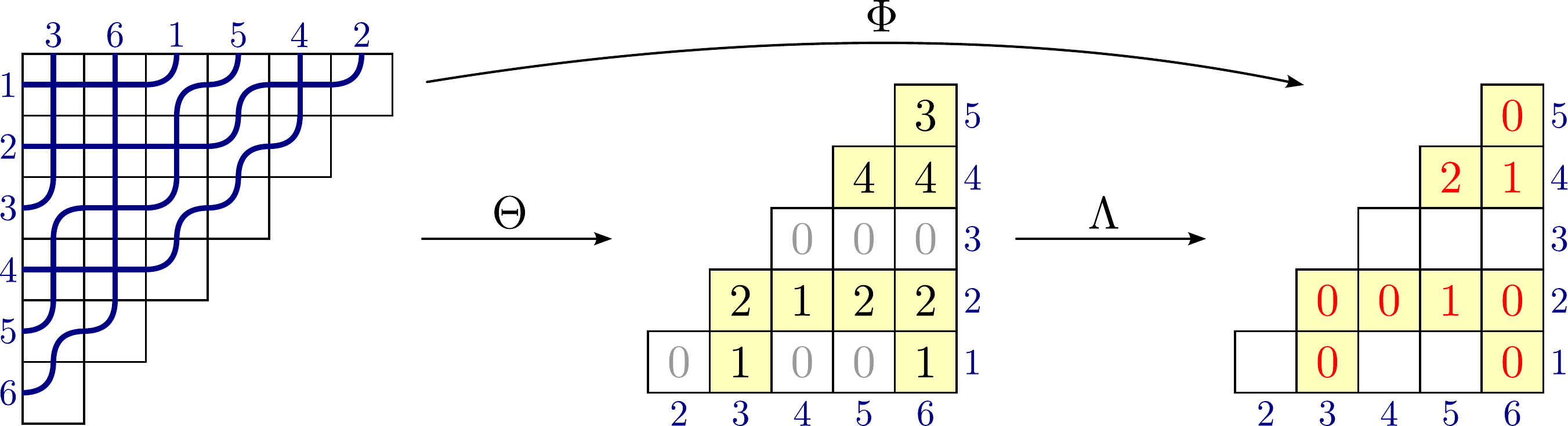}
  \end{center}
\caption{ An illustration of the maps $\Theta$ and $\Lambda$ and their composition $\Phi=\Lambda\circ\Theta$.  }\label{fig:pipe_inversions_Lehmer} 
\end{figure}

We now discuss the operations on inversions tableaux that correspond to chute moves. 

Let $T$ be a column-injective tableau for a permutation $w$. Consider a box ${(i,j)\in\ID(w)}$, and let $T(i,j)=a$. Let $b$ be the smallest integer that is greater than $a$ and does not appear below $(i,j)$ in $T$. If $b$ does not appear in column $j$ of $T$, then we can perform a \dfn{pure increment} to box $(i,j)$ in $T$ by replacing the entry $a$ in box $(i,j)$ with the entry $b$. If $b$ does appear in column $j$ of $T$, then we can perform a \dfn{trade increment} to box $(i,j)$ in $T$ by swapping the locations of the entries $a$ and $b$ in column $j$. In general, we can \dfn{increment} box $(i,j)$ in $T$ by performing either a pure increment or a trade increment to box $(i,j)$. Let $\uparrow_{(i,j)}\!T$ be the tableau obtained from $T$ by incrementing box $(i,j)$; note that $\uparrow_{(i,j)}\!T$ is also a column-injective tableau for $w$. The motivation for this definition comes from the following straightforward lemma. 

\begin{lemma}[{\cite{AF}}]\label{obs:incrementing_adding} 
If $T$ is a column-injective tableau for a permutation $w$, then $\Lambda(\uparrow_{(i,j)}\!T)$ is obtained from $\Lambda(T)$ by adding $1$ to the entry in box $(i,j)$. 
\end{lemma} 

It follows from \cref{obs:incrementing_adding} that incrementing is commutative. Thus, it makes sense to speak about incrementing a multiset $\mathcal M$ of boxes in a column-injective tableau (the number of times each box is incremented is its multiplicity in the multiset). Let $\uparrow_{\mathcal M}\!T$ denote the tableau obtained from $T$ by incrementing $\mathcal M$. For each box $(i,j)$, we obtain the entry $\Lambda(\uparrow_{\mathcal M}\!T)(i,j)$ by adding the multiplicity of $(i,j)$ in $\mathcal M$ to the entry $\Lambda(T)(i,j)$.  

The following proposition shows how chute moves relate to incrementing. 

\begin{proposition}[{\cite[Proposition~6.3]{AF}}]\label{prop:chute} 
Let $P_1,P_2\in\PD(w)$ be reduced pipe dreams for a permutation $w\in\SSS_n$, and let $(\xx,\yy)$ be an inversion of $w$. Let $T_1=\Theta(P_1)$ and $T_2=\Theta(P_2)$. Let $p_0=T_1(\xx,\yy)$, and let $q_0$ be the smallest integer that is greater than $p_0$ and that does not appear below $(\xx,\yy)$ in $T$. Then \[P_2=\C_{\xx,\yy}(P_1)\] if and only if there exists a set $Y\subseteq\{\yy+1,\ldots,n\}$ such that the following conditions hold: 
\begin{itemize}
\item $T_2=\uparrow_{\mathcal B}\!T_1$, where $\mathcal B=\{(\xx,\yy)\}\cup\{(\xx,y):y\in Y\}$; 
\item $T_1(\xx,y)=T_2(\yy,y)=p_0$ and $T_2(\xx,y)=T_1(\yy,y)=q_0$ for every $y\in Y$; 
\item no box in column $\yy$ has entry $q_0$ in $T_1$. 
\end{itemize} 
Moreover, if $P_2=\C_{\xx,\yy}(P_1)$, then $Y$ is the set of indices of the pipes that cross vertically through the rectangle $\RR(P_1,P_2)$ in $P_1$ (and also in $P_2$). 
\end{proposition} 

\begin{example}
\cref{fig:chute_inversions} illustrates \cref{prop:chute} with $w=41865732\in\SSS_8$ and $(\xx,\yy)=(3,5)$. In this example, we have $Y=\{6,8\}$, $p_0=1$, and $q_0=3$. Note that the inversions tableau on the right is obtained from the one on the left by performing a pure increment to box $(3,5)$ and performing trade increments to boxes $(3,6)$ and $(3,8)$. 
\end{example}

\begin{figure}[ht]
  \begin{center}
  \includegraphics[height=4.66cm]{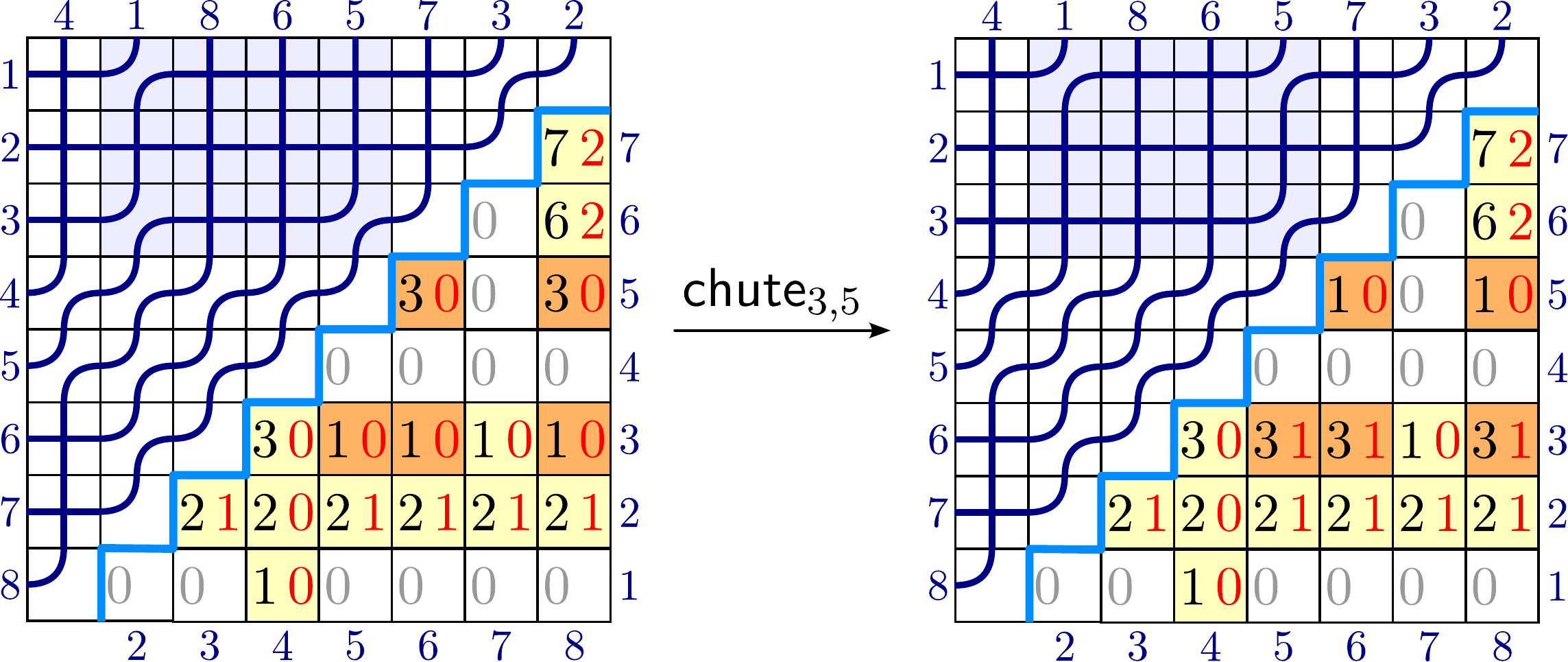}
  \end{center}
\caption{Applying a chute move to a reduced pipe dream, along with its corresponding inversions tableau and Lehmer tableau. Inversions tableaux are drawn with black numbers, while Lehmer tableaux are drawn with {\color{red} red} numbers. The boxes whose entries in the two inversions tableaux differ are {\color{Orange}orange}.}\label{fig:chute_inversions} 
\end{figure}

Suppose $P\leqC P'$. It follows from \cref{prop:chute} that the inversions tableau $\Theta(P')$ is obtained from $\Theta(P)$ by incrementing some multiset of boxes. By \cref{obs:incrementing_adding}, $\Phi(P')$ is obtained from $\Phi(P)$ by increasing the entries in some boxes; that is, $\Phi(P)\leq\Phi(P')$. This yields the following corollary. 

\begin{corollary}[{\cite[Corollary~6.16]{AF}}]\label{cor:easy_direction} 
If $P,P'\in\PD(w)$ are such that ${P\leqC \!P'}$, then ${\Phi(P)\!\leq\!\Phi(P')}$. 
\end{corollary} 

\section{Preparatory Results}\label{sec:lemmas} 

This section is devoted to establishing a series of technical preparatory lemmas. 

\begin{lemma}\label{lem:Colinjective}
Let $T,T'\in\CIT(w)$ be column-injective tableaux for a permutation $w\in\SSS_n$. Fix a box $(i,j)$. Suppose $T'=\uparrow_{\mathcal M}\!T$, where $\mathcal M$ is a multiset of boxes. Let $\RRR$ be the set of positive integers $r$ satisfying the inequalities $r<i$, $T(r,j)<T(i,j)$, and $T'(r,j)>T'(i,j)$. Then ${|\RRR|\geq T(i,j)-T'(i,j)}$. Moreover, if $(i,j)$ is not in $\mathcal M$, then $|\RRR|=T(i,j)-T'(i,j)$. 
\end{lemma} 
\begin{proof} 
Because incrementing a box only changes the entries of boxes in the same column, we may assume without loss of generality that all boxes in $\mathcal M$ are in column $j$. Let $b=T(i,j)$ and $b'=T'(i,j)$. If $\mathcal M$ is empty, then the proof is trivial because $b=b'$. Therefore, we may assume $\mathcal M$ is nonempty and proceed by induction on the size of $\mathcal M$ (i.e., the sum of the multiplicities of elements of $\mathcal M$). If $b<b'$, then the inequality $|\RRR|\geq b-b'$ holds trivially, and the hypothesis of the last statement is not satisfied because $(i,j)$ must be in $\mathcal M$. If $b=b'$, then the inequality $|\RRR|\geq b-b'$ holds trivially, and the last statement follows readily from the definition of an increment. Thus, we may assume $b>b'$. 

Let $i'$ be the smallest integer such that $(i',j)\in\mathcal M$. Each time we increment a box, we do not change the entries lying below that box, and we do not decrease the entry in the box that we increment. Therefore, since $b>b'$, we must have $i'<i$. Let $\mathcal M'$ be the multiset obtained from $\mathcal M$ by decreasing the multiplicity of $(i',j)$ by $1$. Let $T^*=\uparrow_{(i',j)}\!T$, and note that $T'=\uparrow_{\mathcal M'}\!T^*$. Let $b^*=T^*(i,j)$. Let $\RRR^*$ be the set of positive integers $r^*$ satisfying the inequalities $r^*<i$, $T^*(r^*,j)<b^*$, and $T'(r^*,j)>b'$. By induction, we know that $|\RRR^*|\geq b^*-b'$, and we also know that $|\RRR^*|=b^*-b'$ if $(i,j)$ is not in $\mathcal M'$. 

We obtain $T^*$ from $T$ by performing either a pure increment or a trade increment to box $(i',j)$. Regardless of which type of increment is performed, we have $b^*\leq b$, and every box below $(i,j)$ that has an entry less than $b$ in $T^*$ also has an entry less than $b$ in $T$. This implies that $\RRR^*\subseteq\RRR$. We now consider two cases. 

\medskip 

\noindent {\bf Case 1.} Suppose $b^*=b$. 
In this case, $|\RRR|\geq b^*-b'=b-b'$ since $\RRR^*\subseteq\RRR$. To prove the second statement, assume $(i,j)$ is not in $\mathcal M$. Then $(i,j)$ is not in $\mathcal M'$, so $|\RRR^*|=b^*-b'=b-b'$. Because $b=b^*$, it follows from the definition of an increment that every box below $(i,j)$ that has an entry less than $b$ in $T$ also has an entry less than $b$ in $T^*$. This shows that $\RRR=\RRR^*$, so $|\RRR|=b-b'$. 

\medskip 

\noindent {\bf Case 2.} Suppose $b^*<b$. This implies that $T^*$ is obtained from $T$ by performing a trade increment that swaps the entries in boxes $(i,j)$ and $(i',j)$. Hence, $T(i',j)=T^*(i,j)=b^*$. By the definition of a trade increment, all of the integers in the interval $[b^*+1,b-1]$ appear below the box $(i',j)$ in $T$. The boxes below $(i',j)$ are unaffected when we increment boxes in $\mathcal M$, so each integer in $[b^*+1,b-1]$ occupies the same box in $T$, $T^*$, and $T'$. Because $T(i,j)\neq T'(i,j)=b'$, this implies that ${b'\not\in[b^*+1,b-1]}$, so $b'\leq b^*$. This forces every integer $r$ such that $T(r,j)\in[b^*+1,b-1]$ to be in $\RRR\setminus\RRR^*$. Since none of the boxes in $\mathcal M'$ are below $(i',j)$, we know that ${T'(i',j)\geq T^*(i',j)=b>b'}$. Because $T(i',j)=b^*<b$, it follows that $i'\in\RRR$. We also know that $i'\not\in\RRR^*$ because ${T^*(i',j)=b>b^*}$. This shows that all of the integers $r$ such that $T(r,j)\in[b^*,b-1]$ are in $\RRR\setminus\RRR^*$. On the other hand, every integer $r\in \RRR\setminus\RRR^*$ must satisfy $T(r,j)\in[b^*,b-1]$ (by the definitions of $\RRR$ and $\RRR^*$). We have seen that $\RRR^*\subseteq\RRR$, so $|\RRR|= b-b^*+|\RRR^*|\geq (b-b^*)+(b^*-b')=b-b'$. Moreover, if $(i,j)$ is not in $\mathcal M$, then it is also not in $\mathcal M'$, so $|\RRR|= b-b^*+|\RRR^*|= (b-b^*)+(b^*-b')=b-b'$.  
\end{proof} 

\begin{lemma}\label{lem:5.1}
Let $T,T'\in\CIT(w)$ be column-injective tableaux for a permutation $w\in\SSS_n$.  Fix integers $1\leq i<j<k\leq n$, and assume that $T'(i,j)\leq T(i,k)<T(i,j)$. Assume that for every positive integer $r<i$, the $\Le$-shapes $\Le(r,i,j)$ and $\Le(r,i,k)$ are balanced in $T$ and $T'$. Suppose $T'=\uparrow_{\mathcal M}\!T$, where $\mathcal M$ is a multiset of boxes that does not contain $(i,k)$. Then $T'(i,j)>T'(i,k)$. 
\end{lemma} 
\begin{proof}
Let $a=T(i,k)$, $b=T(i,j)$, $a'=T'(i,k)$, and $b'=T'(i,j)$, and assume that $b'\leq a<b$. Because $(i,k)$ is not in $\mathcal M$, we have $a'\leq a$. Suppose by way of contradiction that $b'\leq a'\leq a$. Let $\RRR$ be the set of positive integers $r$ satisfying the inequalities $r<i$, $T(r,j)<b$, and $T'(r,j)>b'$. It follows from \cref{lem:Colinjective} that $|\RRR|\geq b-b'$. 

Let $\QQQ=\{r\in\RRR:T(r,j)\leq a\}$ and $\QQQ'=\{r\in\RRR:T'(r,i)\geq a'\}$. Because the entries in column $j$ of $T$ are distinct, we have $|\RRR\setminus\QQQ|\leq b-a-1$, so $|\QQQ|\geq (b-b')-(b-a-1)=a-b'+1$. For each $r\in\RRR$, the $\Le$-shape $\Le(r,i,j)$ is balanced in $T'$, so $T'(r,i)\geq T'(r,j)>T'(i,j)=b'$. Because the entries in column $i$ of $T'$ are distinct, this implies that $|\RRR\setminus\QQQ'|\leq a'-b'-1$, so we have ${|\QQQ'|\geq (b-b')-(a'-b'-1)=b-a'+1}$. Consequently, 
\begin{align}
\nonumber |\QQQ\cap\QQQ'|&=|\QQQ|+|\QQQ'|-|\QQQ\cup\QQQ'| \\ \nonumber &\geq (a-b'+1)+(b-a'+1)-(b-b') \\ 
\label{eq:QcapQ'} &=a-a'+2.
\end{align} 

Let $\RRR'$ be the set of positive integers $r$ satisfying the inequalities $r<i$, $T(r,k)<a$, and $T'(r,k)>a'$. Because $(i,k)$ is not in $\mathcal M$, it follows from \cref{lem:Colinjective} that $|\RRR'|=a-a'$. We will prove that $\QQQ\cap\QQQ'\subseteq\RRR'$, which will contradict \eqref{eq:QcapQ'}. 

Suppose $r\in\QQQ\cap \QQQ'$. Because $r\in \QQQ$, we have $T(r,j)\leq a<b=T(i,j)$.  Since $\Le(r,i,j)$ is balanced in $T$, we have $T(r,i)\leq T(r,j)<T(i,j)$. Hence, $T(r,i)\leq a$. Because $\Le(r,j,k)$ is also balanced in $T$, we must have $T(r,i)\leq T(r,k)\leq T(i,k)=a$. In fact, we have $T(r,k)<a$ because the entries in column $k$ of $T$ are distinct. Because $r\in \QQQ'$, we have $T'(r,i)\geq a'$.  Since $\Le(r,i,k)$ is balanced in $T$, we must have $T'(r,i)\geq T'(r,k)\geq T'(i,k)=a'$. In fact, we have $T'(r,k)>a'$ because the entries in column $k$ of $T'$ are distinct. We have shown that $T(r,k)<a$ and $T'(r,k)>a'$, so $r\in\RRR'$.   
\end{proof} 

\begin{lemma}\label{lem:less_hook} 
Let $T\in\CIT(w)$ be a column-injective tableaux for a permutation $w\in\SSS_n$. Suppose $T'\in\IT(w)$ is an inversions tableau for $w$ such that $T'=\uparrow_{\mathcal M}\!T$ for some multiset $\mathcal M$ of boxes. Fix integers $1\leq i<k\leq n$ such that $T(i,k)$ is strictly less than the median of the entries in $\hook(i,k)$ in $T$. Assume that for all $r,j$ satisfying $1\leq r<i<j\leq n$, the $\Le$-shape $\Le(r,i,j)$ is balanced in $T$. Then $(i,k)\in\mathcal M$. 
\end{lemma} 
\begin{proof} 
Given a tableau $\tau\in\CIT(w)$, let $\boldsymbol{a}(\tau)$ be the number of entries in $\tau$ that appear above box $(i,k)$ and are greater than $\tau(i,k)$. Similarly, let $\boldsymbol{\ell}(\tau)$ be the number of entries in $\tau$ that appear to the left of box $(i,k)$ and are greater than $\tau(i,k)$. It follows from the definition of an increment that $\boldsymbol{a}(\uparrow_{(r,s)}\!\tau)=\boldsymbol{a}(\tau)$ for every box $(r,s)\neq(i,k)$. Because $T(i,k)$ is strictly less than the median of the entries in $\hook(i,k)$ in $T$, we have $\boldsymbol{a}(T)+\boldsymbol{\ell}(T)>k-i-1$ (note that $|\hook(i,k)|=2(k-i)-1$). 

Suppose by way of contradiction that $(i,k)$ is not in $\mathcal M$. 
Then $\boldsymbol{a}(T')=\boldsymbol{a}(T)$. Because $T'$ is an inversions tableau, $\hook(i,k)$ is balanced in $T'$. Consequently, ${\boldsymbol{a}(T')+\boldsymbol{\ell}(T')\leq k-i-1}$. This shows that $\boldsymbol{\ell}(T')<\boldsymbol{\ell}(T)$, so there exists an integer $j$ satisfying $i<j<k$ such that ${T(i,j)>T(i,k)}$ and $T'(i,j)\leq T'(i,k)$. Because $(i,k)$ is not in $\mathcal M$, we have $T'(i,k)\leq T(i,k)$. Therefore, ${T'(i,j)\leq T(i,k)<T(i,j)}$. Invoking \cref{lem:5.1}, we find that $T'(i,j)>T'(i,k)$, which is a contradiction. 
\end{proof} 

\begin{lemma}\label{lem:bigger_after} 
Let $T,T'\in\IT(w)$ be inversions tableaux for a permutation $w\in\SSS_n$ that agree in all columns other than column $n$. Suppose also that $T'=\uparrow_{\mathcal M}\!T$ for some multiset $\mathcal M$ of boxes in column $n$. Then $T(i,j)\leq T'(i,j)$ for all $(i,j)\in\LRT_{n-1}$. 
\end{lemma}
\begin{proof}
We know by hypothesis that $T(i,j)=T'(i,j)$ for all boxes $(i,j)$ with $j<n$. Now suppose by way of contradiction that there is a box $(i,n)$ in column $n$ such that $T(i,n)>T'(i,n)$. By \cref{lem:Colinjective}, the number of positive integers $r<i$ satisfying $T(r,n)<T(i,n)$ and $T'(r,n)>T'(i,n)$ is at least $T(i,n)-T'(i,n)$. Since $T'$ is column-injective, there must exist a positive integer $r^*<i$ such that $T(r^*,n)<T(i,n)$ and $T'(r^*,n)\geq T(i,n)>T'(i,n)$. Since $\Le(r^*,i,n)$ is balanced in both $T$ and $T'$, we must have $T(r^*,i)\leq T(r^*,n)<T(i,n)$ and $T'(r^*,i)\geq T'(r^*,n)>T'(i,n)$. This shows that 
$T(i,n)>T(r^*,i)=T'(r^*,i)\geq T'(r^*,n)\geq T(i,n)$, which is a contradiction. 
\end{proof}  

\section{Pipe Dream Transformations}\label{sec:transformations} 

In this section, we discuss two useful ways that one can manipulate pipe dreams. 

\subsection{Transposes} 
Let $P^\top$ denote the transpose of a reduced pipe dream $P$. For each $w\in\SSS_n$, the map $P\mapsto P^\top$ is a bijection from $\PD(w)$ to $\PD(w^{-1})$. In fact, this map is a poset anti-isomorphism, meaning that for all $P_1,P_2\in\PD(w)$, we have $P_1\leqC P_2$ if and only if $P_2^\top\leqC P_1^\top$. We will need \cref{prop:transpose}, which provides some information about how this map interacts with the correspondence between reduced pipe dreams and Lehmer tableaux. 

Given a permutation $w\in\SSS_n$, let $\widehat w\in\SSS_{n-1}$ be the permutation obtained from $w$ by deleting the number $n$. Let $P\in\PD(w)$. For each $m\in[n]$, let $\BB_m(P)$ be the set of boxes in row $m$ of $P$ that contain part of pipe $n$. Note that $\BB_n(P)$ consists of the unique box in row $n$, which is filled with an elbow tile. For $1\leq m\leq n-1$, because $P$ is reduced, $\BB_m(P)$ either consists of a single box filled with a cross tile or two adjacent boxes filled with bump or elbow tiles. Let us delete the rightmost box from $\BB_m(P)$ for each $m\in[n]$ and then shift each box to the right of a deleted box one space to the left; the result is a pipe dream $\widehat P$. The following lemma is immediate upon inspection (see also \cref{fig:delete}). 

\begin{lemma}\label{lem:delete} 
Let $P\in\PD(w)$ be a reduced pipe dream for a permutation $w\in\SSS_n$. Then $\widehat P$ is a reduced pipe dream for $\widehat w$. Moreover, $\Theta(\widehat P)$ is obtained from $\Theta(P)$ by deleting column $n$. 
\end{lemma} 

\begin{figure}[ht]
  \begin{center}
  \includegraphics[height=4.66cm]{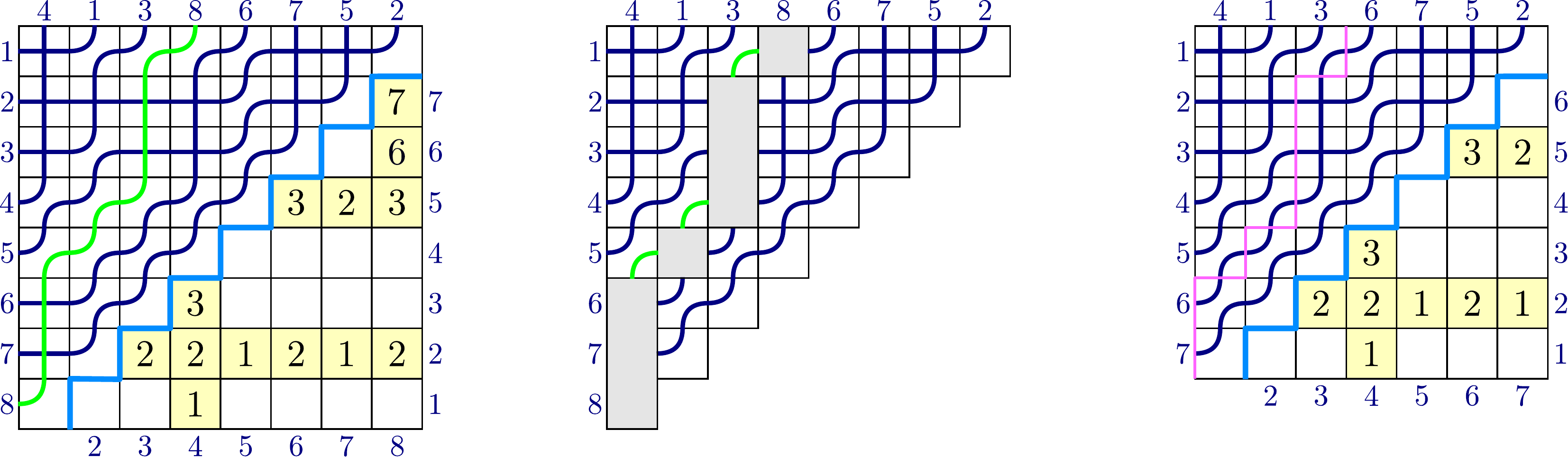}
  \end{center}
\caption{On the left is a pipe dream $P\in\PD(41386752)$ together with its corresponding inversions tableau $\Theta(P)$. In the middle is the result of deleting the rightmost box in $\BB_m(P)$ for each row $m$. On the right is the pipe dream $\widehat P\in\PD(4136752)$.}\label{fig:delete} 
\end{figure}

Given an inversion $(i,j)$ of a permutation $w$ and a reduced pipe dream $P\in\PD(w)$, recall that we write $\row_P(i,j)$ for the index of the row where pipes $i$ and $j$ cross in $P$. Let us also write $\col_P(i,j)$ for the index of the column where pipes $i$ and $j$ cross in $P$ (columns are numbered $1,\ldots,n$ from left to right). It is immediate from the relevant definitions that $\Phi(P)(i,j)=\row_P(i,j)-|A_{i,j}(P)|-1$, where we define $A_{i,j}(P)$ to be the set of pipes indexed by integers less than $i$ that cross pipe $j$ above row $\row_P(i,j)$. Similarly, we have 
\begin{equation}\label{eq:easy}
\Phi(P^\top)(w^{-1}(j),w^{-1}(i))=\col_P(i,j)-|D_{i,j}(P)|-1,
\end{equation}
where we define $D_{i,j}(P)$ to be the set of pipes indexed by elements of the set ${\{w(k):1\leq k<w^{-1}(j)\}}$ that cross pipe $i$ to the left of column $\col_P(i,j)$. The next proposition is illustrated in \cref{fig:transpose2}. 

\begin{figure}[htbp]
  \begin{center}
  \includegraphics[height=10.667cm]{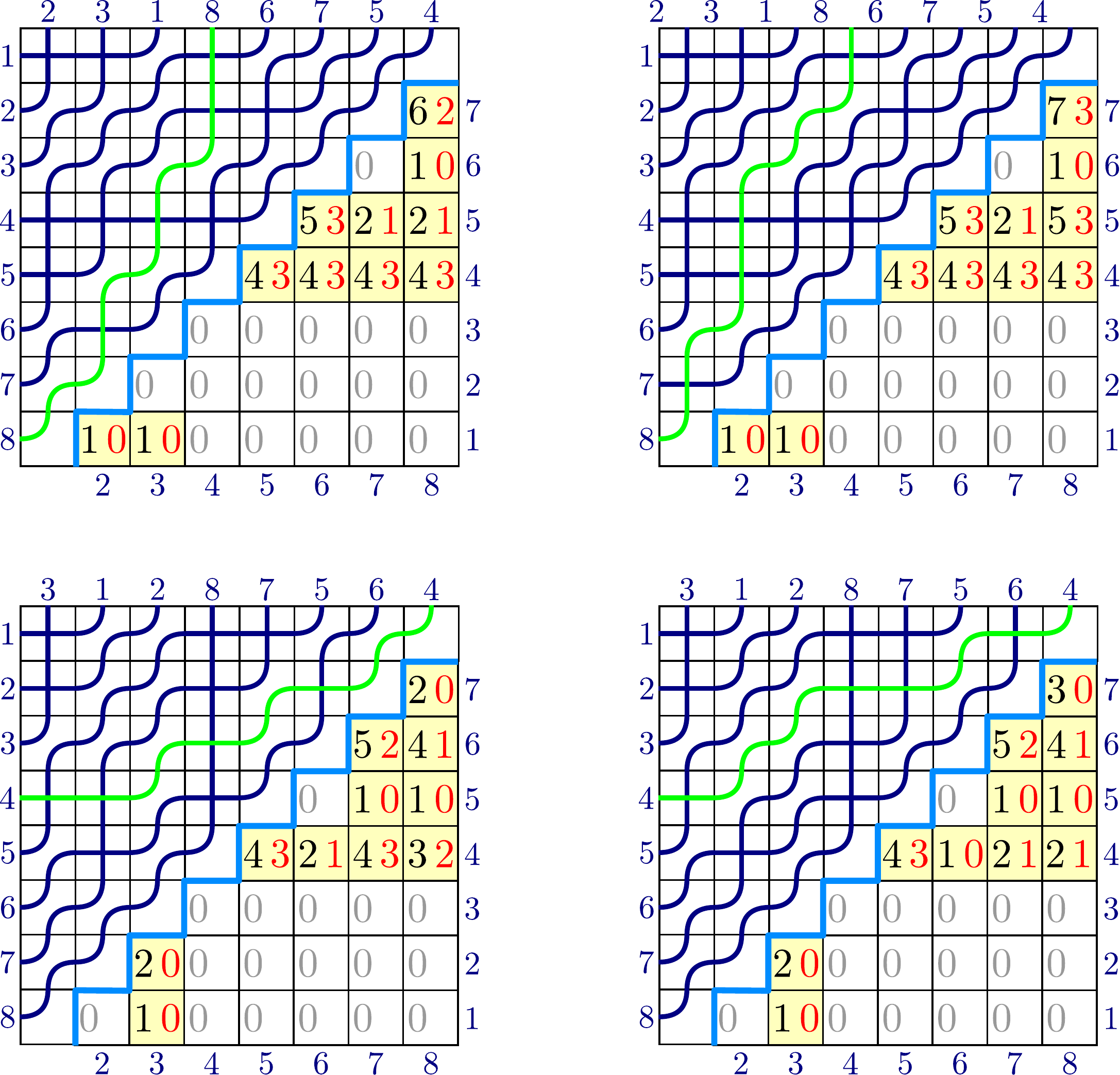}
  \end{center}
\caption{On the top are reduced pipe dreams $P_1$ (left) and $P_2$ (right) together with their inversions tableaux (black) and Lehmer tableaux ({\color{red}red}). We have ${\Phi(P_1)\leq\Phi(P_2)}$, and $\Phi(P_1)$ and $\Phi(P_2)$ disagree only in column $n$. On the bottom are the pipe dreams $P_1^\top$ (left) and $P_2^\top$ (right), also drawn with their inversions tableaux (black) and Lehmer tableaux ({\color{red}red}). We have $\Phi(P_2^\top)\leq\Phi(P_1^\top)$, and $\Phi(P_1^\top)$ and $\Phi(P_2^\top)$ disagree only in row $4$. }\label{fig:transpose2} 
\end{figure}

\begin{proposition}\label{prop:transpose} 
Let $w\in\SSS_n$, and let $P_1,P_2\in\PD(w)$ be such that $\Phi(P_1)\leq\Phi(P_2)$. Suppose that $\Phi(P_1)$ and $\Phi(P_2)$ agree in all columns other than column $n$. Then $\Phi(P_1^\top)$ and $\Phi(P_2^\top)$ agree in all rows other than row $w^{-1}(n)$, and $\Phi(P_2^\top)\leq\Phi(P_1^\top)$.  
\end{proposition} 

\begin{proof}
The inversions tableaux $\Theta(P_1)$ and $\Theta(P_2)$ agree in all columns other than column $n$. By \cref{lem:delete}, this implies that $\Theta(\widehat P_1)=\Theta(\widehat P_2)$, so $\widehat P_1=\widehat P_2$. 

Suppose $(w^{-1}(j),w^{-1}(i))$ is an inversion of $w^{-1}$ such that $j\neq n$. Then $(i,j)$ is an inversion of $w$. For ${\delta\in\{1,2\}}$, we know by \eqref{eq:easy} that 
\[\Phi(P_\delta^\top)(w^{-1}(j),w^{-1}(i))=\col_{P_\delta}(i,j)-|D_{i,j}(P_\delta)|-1\] and 
\[\Phi(\widehat P_\delta^\top)(\widehat w^{-1}(j),\widehat w^{-1}(i))=\col_{P_\delta}(i,j)-|D_{i,j}(\widehat P_\delta)|-1.\]
Note that $D_{i,j}(\widehat P_\delta)=D_{i,j}(P_\delta)\setminus\{n\}$. We claim that 
\begin{equation}\label{eq:hat}
\Phi(P_\delta^\top)(w^{-1}(j),w^{-1}(i))=\Phi(\widehat P_\delta^\top)(\widehat w^{-1}(j),\widehat w^{-1}(i)).
\end{equation} If $n\not\in D_{i,j}(P_\delta)$, then we can use the fact that pipes $j$ and $n$ cannot cross more than once in $P_\delta$ to find that pipe $n$ does not pass through row $\row_{P_\delta}(i,j)$ to the left of column $\col_{P_\delta}(i,j)$. In this case, we have $\col_{P_\delta}(i,j)=\col_{\widehat P_\delta}(i,j)$ and $D_{i,j}(P_\delta)=D_{i,j}(\widehat P_\delta)$, so \eqref{eq:hat} holds. 
On the other hand, if $n\in D_{i,j}(P_\delta)$, then pipe $n$ passes through row $\row_{P_\delta}(i,j)$ to the left of column $\col_{P_\delta}(i,j)$, so $\col_{P_\delta}(i,j)=\col_{\widehat P_\delta}(i,j)+1$. In this case, $D_{i,j}(P_\delta)=D_{i,j}(\widehat P_\delta)\sqcup\{n\}$, so \eqref{eq:hat} holds again. 

For $j\neq n$, we have proven \eqref{eq:hat} for each $\delta\in\{1,2\}$; since $\widehat P_1=\widehat P_2$, this proves that \[\Phi(P_1^\top)(w^{-1}(j),w^{-1}(i))=\Phi(P_2^\top)(w^{-1}(j),w^{-1}(i)).\] It follows that $\Phi(P_1^\top)$ and $\Phi(P_2^\top)$ agree in all rows other than row $w^{-1}(n)$. 

Now fix $i$ such that $(w^{-1}(n),w^{-1}(i))$ is an inversion of $w^{-1}$. For $\delta\in\{1,2\}$, let $C_\delta(i)$ be the set of columns in $P_\delta$ to the left of $\col_{P_\delta}(i,n)$ in which pipe $i$ does not cross with a pipe in $D_{i,n}(P_\delta)$. We have $|C_\delta(i)|=\col_{P_\delta}(i,n)-|D_{i,n}(P_\delta)|-1$. We will prove that $C_2(i)\subseteq C_1(i)$. By \eqref{eq:easy}, this will imply that \[\Phi(P_2^\top)(w^{-1}(n),w^{-1}(i))\leq\Phi(P_1^\top)(w^{-1}(n),w^{-1}(i));\] as $i$ was arbitrary, this will complete the proof that $\Phi(P_2^\top)\leq\Phi(P_1^\top)$. 

For $\delta\in\{1,2\}$ and $m\in[n]$, let \[f_\delta(m)=|\{k\in[n-1]:\row_{P_\delta}(k,n)\leq m\}|.\] For each $m\in[n]$, the rightmost box in $P_\delta$ that contains part of pipe $n$ is $(n-m-f_\delta(m),n)$. According to \cref{lem:bigger_after}, we have $\Theta(P_1)(k,n)\leq \Theta(P_2)(k,n)$ for all $k\in[n-1]$. Equivalently, ${\row_{P_1(k,n)}\leq\row_{P_2}(k,n)}$ for every $k\in[n-1]$. This implies that $f_1(m)\geq f_2(m)$. Therefore, the box in $\BB_m(P_1)$ that we delete when constructing $\widehat P_1$ from $P_1$ lies weakly to the right of the box in $\BB_m(P_2)$ that we delete when constructing $\widehat P_2$ from $P_2$. Because $\widehat P_1=\widehat P_2$, it follows that $C_2(i)\subseteq C_1(i)$.  
\end{proof} 

\subsection{The Triforce Embedding}\label{subsec:Triforce}
Given a permutation $w\in\SSS_n$, define $w^{\Tri}\in\SSS_{2n}$ by 
\[w^{\Tri}(i)=\begin{cases}   i & \text{if } 1\leq i\leq n; \\
    2n+1-w(2n+1-i) & \text{if } n+1\leq i\leq 2n.
    \end{cases}\] 
Given a pipe dream $P\in\PD(w)$, we obtain a pipe dream $P^{\Tri}\in\PD(w^{\Tri})$ as follows. For $i,j\in[n]$ with $i+j\leq n$, we fill box $(n+1-j,n+1-i)$ in $P^{\Tri}$ with the same type of tile that appears in box $(i,j)$ in $P$. As usual, the boxes on the southeast boundary of $\ULT_{2n}$ are filled with elbow tiles in $P^{\Tri}$. All other boxes are filled with bump tiles in $P^{\Tri}$. In other words, to construct $P^{\Tri}$ from $P$, we change the elbow tiles in $P$ into bump tiles, reflect across a line of slope $1$, and then add additional bump and elbow tiles as necessary. We illustrate this construction in \cref{fig:Triforce}. 

\begin{figure}[ht]
  \begin{center}
  \includegraphics[height=6.447cm]{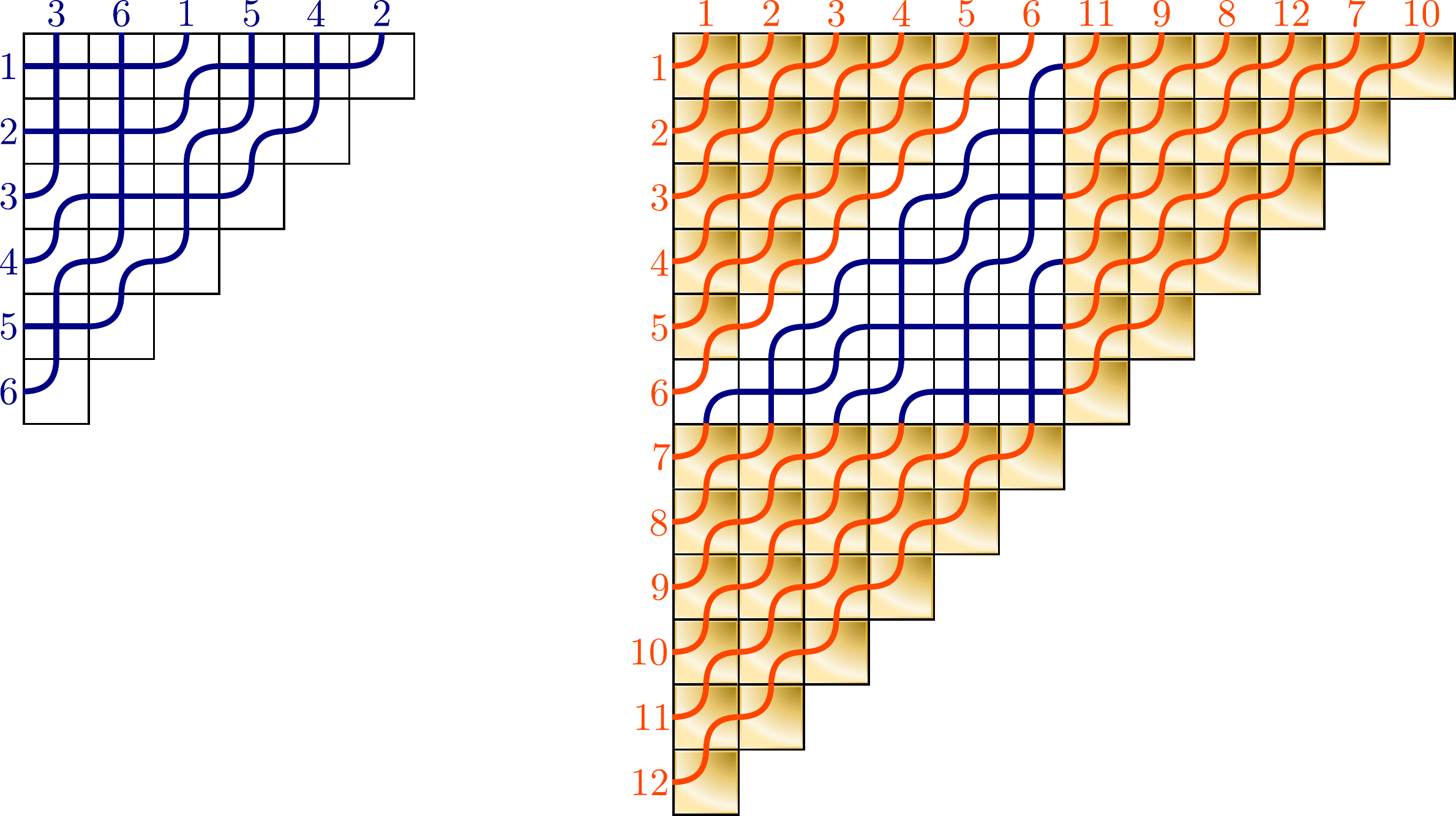}
  \end{center}
\caption{On the left is a reduced pipe dream $P$ for the permutation ${w=361542}$. On the right is the reduced pipe dream $P^{\Tri}$ for the permutation ${w^{\Tri}=1\,2\,3\,4\,5\,6\,11\,9\,8\,12\,7\,10}$. }\label{fig:Triforce} 
\end{figure}

\begin{proposition}\label{prop:Triforce}
Let $w\in\SSS_n$. The map $P\mapsto P^{\Tri}$ is a poset isomorphism from $\PD(w)$ to an interval in $\PD(w^{\Tri})$. 
\end{proposition} 
\begin{proof}
For $P_1,P_2\in\PD(w)$, it is immediate from the definition of a chute move that $P_1\leqC P_2$ if and only if $P_1^{\Tri}\leqC P_2^{\Tri}$. It is also straightforward to see that $\{P^{\Tri}:P\in\PD(w)\}$ is an order-convex subset of $\PD(w^{\Tri})$. This set is in fact an interval because $\PD(w)$ has a minimum element and a maximum element (by \cite[Theorem~3.8]{Rubey}). 
\end{proof}

We call the map $P\mapsto P^{\Tri}$ the \dfn{triforce embedding} of $\PD(w)$ into $\PD(w^{\Tri})$.

\section{The Isomorphism}\label{sec:isomorphism}  

The goal of this section is to prove that the bijection $\Phi\colon\PD(w)\to\LT(w)$ is a poset isomorphism.

\subsection{Constructing Chute Moves} 

For the following proposition, recall that $\uparrow_{(i,j)}\!T_{\leq h}$ denotes the tableau formed by incrementing the box $(i,j)$ in the restricted tableau $T_{\leq h}$.

\begin{proposition}\label{prop:construct} 
Let $T,T'\in\IT(w)$ be inversions tableaux for a permutation $w\in\SSS_n$, and assume that $T'=\uparrow_{\mathcal M}\!T$ for some multiset $\mathcal M$ of boxes. Let $\mathcal X$ be the set of boxes $(i,j)$ in $\mathcal M$ such that $\uparrow_{(i,j)}\!T_{\leq j}$ is an inversions tableau. Assume $\mathcal X$ is nonempty, and choose $(\xx,\yy)\in\mathcal X$ such that no other box in $\mathcal X$ appears to the right of $(\xx,\yy)$ in the same row. Let $Y$ be the set of integers $y$ such that $y_0<y\leq n$ and $T(\xx,y)=T(\xx,\yy)<T(\yy,y)$. Let $\mathcal B=\{(\xx,\yy)\}\cup\mathcal B'$, where $\mathcal B'=\{(\xx,y):y\in Y\}$. Let $p_0=T(\xx,\yy)$, and let $q_0$ be the smallest integer that is greater than $p_0$ and that does not appear below $(\xx,\yy)$ in $T$. Then 
\begin{enumerate}[label=\normalfont(\Roman*), ref=\Roman*]
    \item \label{II} $T(\xx,y)=p_0$ and $T(\yy,y)=(\uparrow_{(\xx,y)}\!T)(\xx,y)=q_0$ for every $y\in Y$; 
    \item \label{III} $\mathcal B\subseteq\mathcal M$; 
    \item \label{IV} $\uparrow_{\mathcal B}\!T$ is an inversions tableau; 
    \item \label{I} no box in column $\yy$ has entry $q_0$ in $T$. 
\end{enumerate} 
\end{proposition} 

Throughout this subsection, we preserve the notation from \cref{prop:construct}. This entire subsection is devoted to proving this proposition, which we do by induction on $n-\yy$.

For the base case of the induction, suppose $n=\yy$. In this case, condition \eqref{II} holds vacuously because $Y=\emptyset$. Moreover, $\mathcal B=\{(\xx,\yy)\}\subseteq\mathcal M$, and $\uparrow_{(\xx,\yy)}\!T=\uparrow_{(\xx,\yy)}\!T_{\leq\yy}$ is an inversions tableau because $(\xx,\yy)\in\mathcal X$. Thus, conditions \eqref{III} and \eqref{IV} hold. To prove condition \eqref{I}, suppose instead that $q_0$ appears in column $n$; by the definition of $q_0$, it must appear in some box $(i,n)$ with $i>\xx$. Because $\Le(\xx,i,n)$ is balanced in $T$ and $T(i,n)=q_0>p_0=T(\xx,n)$, we must have $T(\xx,i)\leq p_0$. Because $\Le(\xx,i,n)$ is balanced in $\uparrow_{(\xx,\yy)}\!T$ and $(\uparrow_{(\xx,\yy)}\!T)(i,n)=p_0<q_0=(\uparrow_{(\xx,\yy)}\!T)(\xx,n)$, we must have $(\uparrow_{(\xx,\yy)}\!T)(\xx,i)\geq q_0$. This is a contradiction because $(\uparrow_{(\xx,\yy)}\!T)(\xx,i)=T(\xx,i)$. 

Throughout the remainder of this subsection, we may assume $n>\yy$. Let $Y_{<n}=\{y\in Y:y<n\}$ and $\mathcal B_{<n}=\{(\xx,y):y\in Y_{<n}\}$. Let $\mathcal M_{<n}$ be the multiset of boxes obtained from $\mathcal M$ by deleting the boxes in column $n$. 

\begin{lemma}\label{lem:induction} 
No box in column $\yy$ has entry $q_0$ in $T$. For every $y\in Y_{<n}$, we have \[{T(\xx,y)=p_0}\quad\text{and}\quad {T(\yy,y)=(\uparrow_{(\xx,y)}\!T)(\xx,y)=q_0}.\] We have $\mathcal B_{<n}\subseteq\mathcal M_{<n}$, and $\uparrow_{\mathcal B_{<n}}\!T_{\leq n-1}$ is an inversions tableau. 
\end{lemma} 

\begin{proof}
This is a direct consequence of our induction hypothesis. Indeed, $T_{\leq n-1}$ and $T'_{\leq n-1}$ are inversions tableaux, and $T'_{\leq n-1}=\uparrow_{\mathcal M_{<n}}\!T_{\leq n-1}$. Moreover, column $\yy$ in $T_{\leq n-1}$ is identical to column $\yy$ in $T$. 
\end{proof} 

\cref{lem:induction} completes the proof of condition \eqref{I}. We now aim to prove the following statements.  

\begin{enumerate}[label=\normalfont(\Roman*'), ref=\Roman*']
    \item \label{II'} If $n\in Y$, then $T(\yy,n)=(\uparrow_{(\xx,n)}\!T)(\xx,n)=q_0$. 
    \item \label{III'} If $n\in Y$, then $(\xx,n)\in\mathcal M$. 
    \item \label{IV'} Every $\Le$-shape intersecting column $n$ is balanced in $\uparrow_{\mathcal B}\!T$.
\end{enumerate}

Note that if $n\in Y$, then $T(\xx,n)=p_0$ by the definition of $Y$. Therefore, in light of \cref{lem:induction}, statements \eqref{II'} and \eqref{III'} imply \eqref{II} and \eqref{III} from \cref{prop:construct}, and \eqref{IV'} implies that $\uparrow_{\mathcal B}\!T$ is balanced. Once we have proven \eqref{III}, we will know that all the entries in each row $i$ of $\uparrow_{\mathcal B}\!T$ are at most $i$ (since $T'$ can be obtained from $\uparrow_{\mathcal B}\!T$ by incrementing a multiset of boxes), so it will follow that $\uparrow_{\mathcal B}\!T$ is an inversions tableau, which is condition \eqref{IV}. In summary, proving \eqref{II'}, \eqref{III'}, and \eqref{IV'} will complete the proof of \cref{prop:construct}. 

\begin{lemma}\label{lem:h_t} 
Suppose $t\in[p_0+1,q_0-1]$ is an integer. There is a positive integer $h_t<\xx$ such that $T(h_t,\xx)=T(h_t,\yy)=t$. Moreover, if $T(\xx,n)=p_0$, then $T(h_t,n)=t$. 
\end{lemma}
\begin{proof}
It is immediate from the definition of $q_0$ that there is a positive integer $h_t<\xx$ such that $T(h_t,\yy)=t$. Because $\Le(h_t,\xx,\yy)$ is balanced in $T$ and $T(\xx,\yy)=p_0<t=T(h_t,\yy)$, we must have $T(h_t,\xx)\geq t$. Because $\Le(h_t,\xx,\yy)$ is balanced in $\uparrow_{(\xx,\yy)}\!T$ (since $\uparrow_{(\xx,\yy)}\!T_{\leq\yy}$ is an inversions tableau) and ${(\uparrow_{(\xx,\yy)}\!T)(\xx,\yy)=q_0>t=(\uparrow_{(\xx,\yy)}\!T)(h_t,\yy)}$, we must have $(\uparrow_{(\xx,\yy)}\!T)(h_t,\xx)\leq t$. But $(\uparrow_{(\xx,\yy)}\!T)(h_t,\xx)=T(h_t,\xx)$, so $T(h_t,\xx)=t$. 

Now suppose $T(\xx,n)=p_0$. We will prove by induction on $t$ that $T(h_t,n)=t$. Thus, we may assume that $T(h_{t'},n)=t'$ for every integer $t'\in[p_0+1,t-1]$. Since $T$ is column-injective, this implies that $T(h_t,n)\not\in[p_0+1,t-1]$. Similarly, $T(h_t,n)\neq T(\xx,n)=p_0$. Because $\Le(h_t,\xx,n)$ is balanced in $T$ and $T(\xx,n)=p_0<t=T(h_t,\xx)$, we must have $p_0\leq T(h_t,n)\leq t$. This implies that $T(h_t,n)=t$. 
\end{proof} 

Throughout the remainder of this section, for each integer $t\in[p_0+1,q_0-1]$, we let $h_t$ be as in \cref{lem:h_t}. 

\begin{lemma}\label{lem:notinrow}
No integer in the interval $[p_0+1,q_0-1]$ appears in row $\xx$ or row $\yy$ in $T$. 
\end{lemma}
\begin{proof}
Suppose $t\in[p_0+1,q_0-1]$ is an integer. We will prove that $t$ does not appear in row $\xx$ in $T$; a similar argument shows that $t$ does not appear in row $\yy$ in $T$. Suppose instead that $T(\xx,j)=t$ for some $j>\xx$. By \cref{lem:h_t}, we have $T(h_t,\xx)=T(\xx,j)=t$. Since $\Le(h_t,\xx,j)$ is balanced in $T$, we have $T(h_t,j)=t=T(\xx,j)$. This is a contradiction because $T$ is column-injective. 
\end{proof}

The next lemma proves \eqref{III'} and makes progress toward proving \eqref{II'}. 

\begin{lemma}\label{lem:ifninY1}
Suppose $n\in Y$. Then $(\xx,n)\in\mathcal M$, and $(\uparrow_{(\xx,n)}\!T)(\xx,n)=q_0$. 
\end{lemma} 
\begin{proof}
Let $T^\#=\uparrow_{\mathcal B_{<n}}\!T$. Note that $T(\xx,\yy)<T^\#(\xx,\yy)$ since $(\xx,\yy)\in\mathcal B_{<n}$. \cref{lem:induction} tells us that $\uparrow_{\mathcal B_{<n}}\!T_{\leq n-1}$ is balanced, so all $\Le$-shapes that do not intersect column $n$ are balanced in $T^\#$. The tableaux $T$ and $T^\#$ agree in column $n$ and in every row below row $\xx$. Since $T$ is balanced, it follows that for all $1\leq r<\xx<j\leq n$, the $\Le$-shape  $\Le(r,\xx,j)$ is balanced in $T^\#$. For each box of the form $(j,n)$ in $\hook(\xx,n)$, we have $T(j,n)=T^\#(j,n)$. For each box of the form $(\xx,j)$ in $\hook(\xx,n)$, we have $T(\xx,j)\leq T^\#(\xx,j)$. Because $n\in Y$, we also know that $T^\#(\xx,n)=p_0<T^\#(\yy,n)$ and $T^\#(\xx,n)=p_0<q_0=T^\#(\xx,\yy)$. Since $\hook(\xx,n)$ is balanced in $T$, this implies that $T^\#(\xx,n)$ is strictly less than the median of the entries in $\hook(i,n)$ in $T^\#$. Therefore, we can apply \cref{lem:less_hook} (with $i=\xx$ and $k=n$ and with $T^\#$ playing the role of $T$) to find that $(\xx,n)\in\mathcal M$. 

Because $n\in Y$, we have $T(\xx,n)=p_0$. For each integer $t\in[p_0+1,q_0-1]$, we know by \cref{lem:h_t} that there is a positive integer $h_t<\xx$ such that $T(h_t,\xx)=T(h_t,\yy)=T(h_t,n)=t$. This implies that $(\uparrow_{(\xx,n)}\!T)(\xx,n)\geq q_0$. Suppose by way of contradiction that $(\uparrow_{(\xx,n)}\!T)(\xx,n)>q_0$. Then there is a positive integer $h_{q_0}<\xx$ such that $T(h_{q_0},n)=q_0$. All of the integers in $[p_0+1,q_0]$ appear below $(\xx,n)$ in $T$, and we have $T(\yy,n)>T(\xx,n)=p_0$ because $n\in Y$. Consequently, $T(\yy,n)>q_0$. The $\Le$-shape $\Le(h_{q_0},\yy,n)$ (indicated by a {\color{Purple}purple} rectangle in \cref{fig:A1}) is balanced in $T$, so $T(h_{q_0},\yy)<q_0$. Since the entries $p_0,p_0+1,\ldots,q_0-1$ appear in column $\yy$ of $T$ in rows $\xx,h_{p_0+1},\ldots,h_{q_0-1}$, we deduce that $T(h_{q_0},\yy)<p_0$. The $\Le$-shape $\Le(h_{q_0},\xx,\yy)$ (indicated by a {\color{MildGreen}green} rectangle in \cref{fig:A1}) is balanced in $T$, so $T(h_{q_0},\xx)<p_0$. However, this contradicts the fact that $\Le(h_{q_0},\xx,n)$ is balanced in $T$. 
\end{proof}

\begin{figure}[ht]
  \begin{center}
  \includegraphics[height=5.164cm]{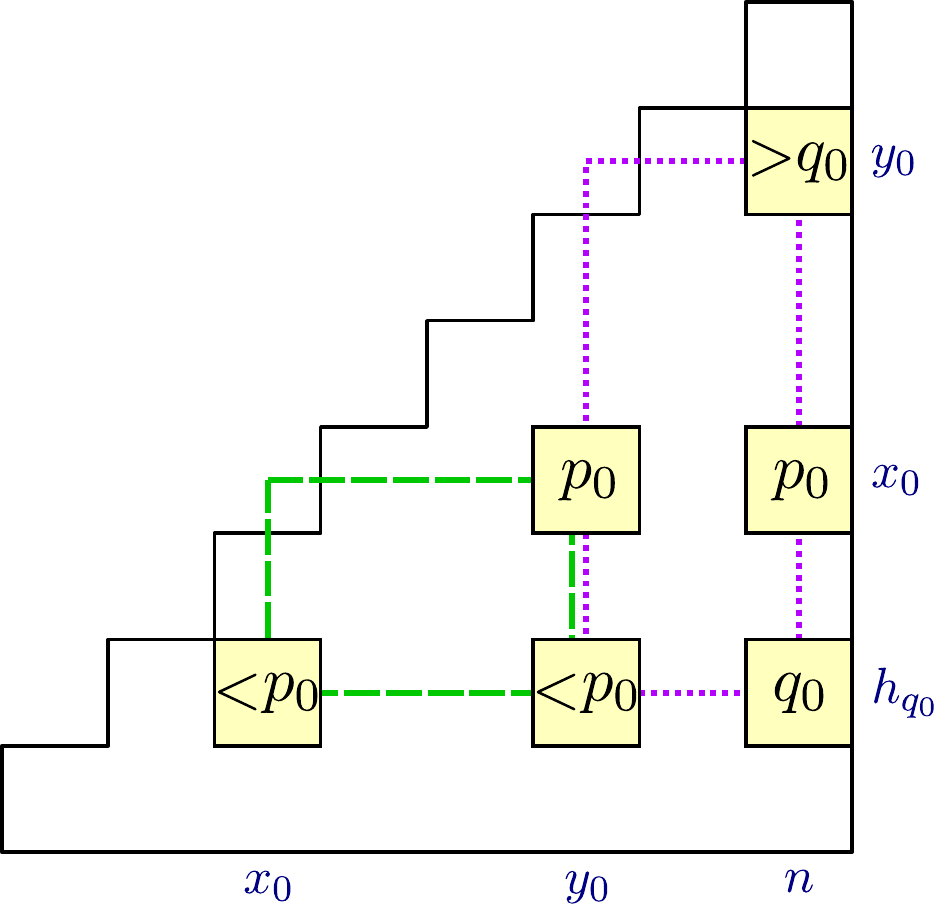}
  \end{center}
\caption{A schematic illustration of some entries in the tableau $T$ in the proof of \cref{lem:ifninY1}. By assuming that $T(h_{q_0},n)=q_0$, we are led to the contradiction that $\Le(h_{q_0},\xx,n)$ is not balanced. }\label{fig:A1} 
\end{figure}

\begin{lemma}\label{lem:ifninY2} 
Suppose $n\in Y$. Then $q_0$ appears in column $n$ of $T$. 
\end{lemma} 
\begin{proof}
Let $T^\#=\uparrow_{(\xx,n)}\!T$. It follows from the definition of the box $(\xx,\yy)$ that $(\xx,n)\not\in\mathcal X$, so $T^\#$ is not an inversions tableau. We know by \cref{lem:ifninY1} that $(\xx,n)\in\mathcal M$, so for each $i$, the entries in row $i$ of $T^\#$ are at most $i$. This implies that $T^\#$ is not balanced. 

Suppose by way of contradiction that $q_0$ does not appear in column $n$ of $T$. Then by \cref{lem:ifninY1}, the tableau $T^\#$ is obtained from $T$ by replacing the entry $p_0$ in box $(\xx,n)$ with the entry $q_0$. Since $T$ is balanced but $T^\#$ is not, there is a $\Le$-shape containing $(\xx,n)$ that is not balanced in $T^\#$. We consider two cases. 

\medskip 

\noindent {\bf Case 1.} Suppose there exists $i<\xx$ such that $\Le(i,\xx,n)$ is not balanced in $T^\#$. Let $t=T(i,n)$. Since $T^\#(i,\xx)=T(i,\xx)$ and $T^\#(i,n)=T(i,n)=t$, we must have $t\in[p_0+1,q_0-1]$. By \cref{lem:h_t}, $i=h_t$. But then $T(i,\xx)=t$, so $T^\#(i,\xx)=T^\#(i,n)$. Hence, $\Le(i,\xx,n)$ is balanced in $T^\#$, which is a contradiction. 

\medskip 

\noindent {\bf Case 2.} Suppose there exists $i$ satisfying $\xx<i<n$ such that $\Le(\xx,i,n)$ is not balanced in $T^\#$. We know by \cref{lem:h_t} that the integers in $[p_0+1,q_0-1]$ appear below $(\xx,n)$ in $T$, and we have assumed that $q_0$ does not appear in column $n$ of $T$, so either $T(i,n)<p_0$ or $T(i,n)>q_0$. Also, because $n\in Y$, we have $T(\yy,n)>p_0$, so in fact $T(\yy,n)>q_0$. If $T(i,n)>q_0$, then because $\Le(\xx,i,n)$ is balanced in $T$, it is also balanced in $T^\#$. This is impossible, so we must have $T(i,n)<p_0$. Since $\Le(\xx,i,n)$ is balanced in $T$, this implies that $T(\xx,i)\geq p_0$. Since $\Le(\xx,i,n)$ is not balanced in $T^\#$ and $T^\#(\xx,i)=T(\xx,i)$, we have $T(\xx,i)<q_0$. Thus, $p_0\leq T(\xx,i)<q_0$. But \cref{lem:notinrow} tells us that $T(\xx,i)\not\in[p_0+1,q_0-1]$, so $T(\xx,i)=p_0$. Because $n\in Y$, we have $T(\yy,n)>p_0>T(i,n)$, so $i\neq\yy$. We now consider two subcases depending on whether $i<\yy$ or $i>\yy$. 

Assume first that $i<\yy$. (See the left side of \cref{fig:A2}.) Since $\Le(i,\yy,n)$ is balanced in $T$, we must have $T(i,\yy)\leq T(i,n)<p_0$. Let $T^\dd=\uparrow_{\mathcal B_{<n}}T$. 
Now consider the $\Le$-shape $\Le(\xx,i,\yy)$, which is balanced in $T^\dd$ by \cref{lem:induction}. Since $i\not\in\mathcal B_{<n}$, we have $T^\dd(\xx,i)=T(\xx,i)<q_0$. Since $q_0$ does not appear in column $\yy$ of $T$ (by \cref{lem:induction}), we have $T^\dd(i,\yy)=T(i,\yy)<p_0$. But this contradicts the fact that $\Le(\xx,i,\yy)$ is balanced in $T^\dd$ because $T^\dd(\xx,\yy)=q_0$. 

Next, assume that $i>\yy$. (See the right side of \cref{fig:A2}.) We know that $\Le(\yy,i,n)$ is balanced in $T$ and that $T(i,n)<p_0<q_0<T(\yy,n)$ (since $n\in Y$), so we must have $T(\yy,i)>q_0$. We have shown that $T(\xx,i)=T(\xx,\yy)=p_0<T(\yy,i)$, so $i\in Y$. But then \cref{lem:induction} tells us that $T(\yy,i)=q_0$, which is a contradiction. 
\end{proof} 

\begin{figure}[hbpt] 
  \begin{center}
  \includegraphics[height=6.114cm]{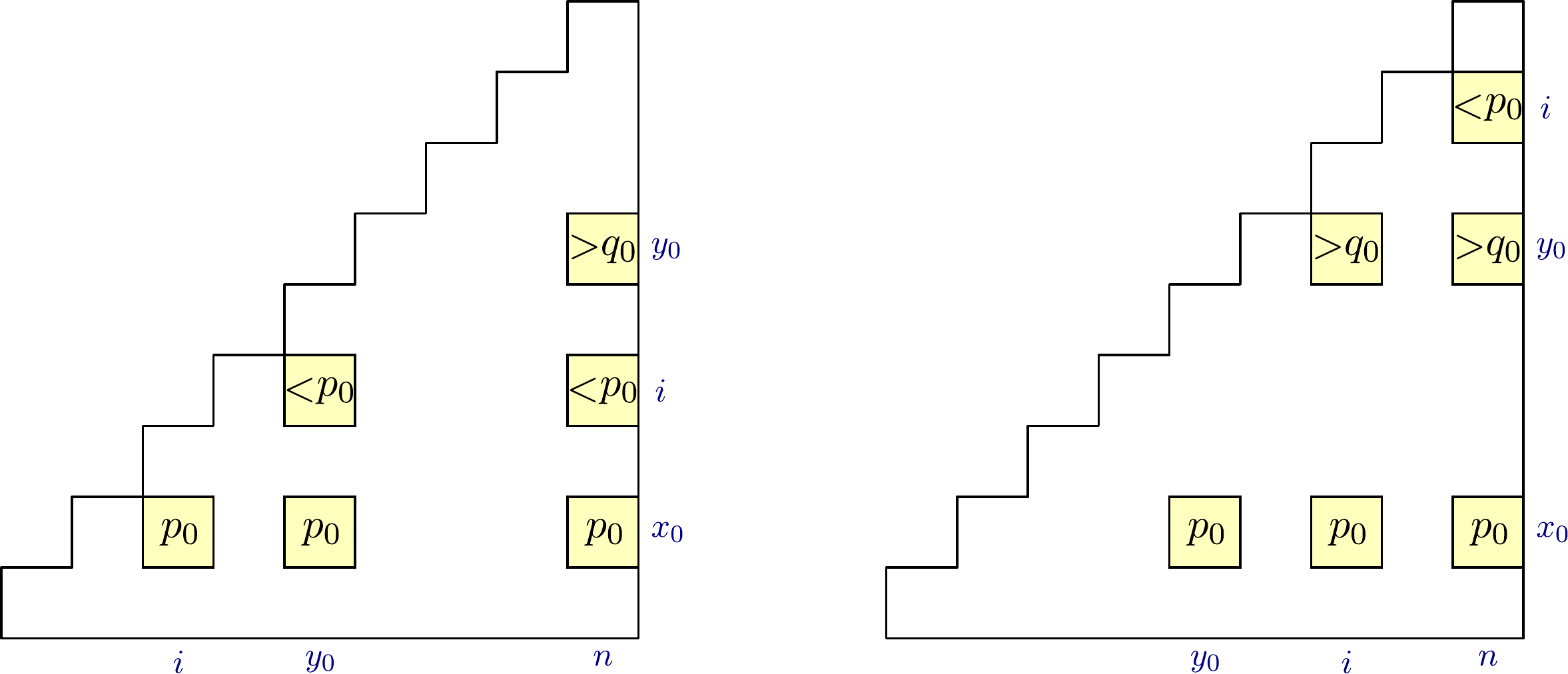}
  \end{center}
\caption{A schematic illustration of some entries in the tableau $T$ in Case~2 of the proof of \cref{lem:ifninY2}. On the left is the case where $i<\yy$; on the right is the case where $i>\yy$.  }\label{fig:A2} 
\end{figure} 

\begin{lemma}\label{lem:ifninY3}
Suppose $n\in Y$. Then $T(\yy,n)=q_0$. 
\end{lemma} 

\begin{proof}
We know by \cref{lem:ifninY2} that there is some row $i$ such that $T(i,n)=q_0$, and we know by \cref{lem:ifninY1} that $i>\xx$. Suppose by way of contradiction that $i\neq\yy$. Because $n\in Y$, we have $T(\yy,n)>p_0$. We also know by \cref{lem:h_t} that all of the integers in $[p_0+1,q_0-1]$ appear below $(\xx,n)$ in $T$, so $T(\yy,n)>q_0$. We now consider two cases. 

\medskip 

\noindent {\bf Case 1.} Suppose $\xx<i<\yy$. (See the left side of \cref{fig:A3}.) Because $\Le(i,\yy,n)$ is balanced in $T$, we have $T(i,\yy)\leq T(i,n)=q_0$. \cref{lem:induction} tells us that $q_0$ does not appear in column $\yy$ of $T$,
so $T(i,\yy)<q_0$. Because $\Le(\xx,i,n)$ is balanced in $T$ and $T(i,n)>T(\xx,n)=p_0$, we have $T(\xx,i)\leq p_0$. The definition of $(\xx,\yy)$ ensures that $\uparrow_{(\xx,\yy)}\!T_{\leq\yy}$ is an inversions tableau, so in particular, $\Le(\xx,i,\yy)$ is balanced in $\uparrow_{(\xx,\yy)}\!T_{\leq\yy}$. The boxes $(\xx,i)$ and $(i,\yy)$ have the same entries in $\uparrow_{(\xx,\yy)}\!T_{\leq\yy}$ as in $T$, so $(\uparrow_{(\xx,\yy)}\!T_{\leq\yy})(\xx,\yy)$ lies weakly between $T(\xx,i)$ and $T(i,\yy)$. We have $T(\xx,i)\leq p_0<q_0$ and $T(i,\yy)<q_0$, so $(\uparrow_{(\xx,\yy)}\!T_{\leq\yy})(\xx,\yy)<q_0$. This is a contradiction because $(\uparrow_{(\xx,\yy)}\!T_{\leq\yy})(\xx,\yy)=q_0$ by the definition of $q_0$.   

\medskip 

\noindent {\bf Case 2.} Suppose $\yy<i$. (See the right side of \cref{fig:A3}.) Because $\Le(\yy,i,n)$ is balanced in $T$, we have $T(\yy,i)\geq T(\yy,n)>q_0$. Because $\Le(\xx,\yy,i)$ is balanced in $T$ and $T(\xx,\yy)=p_0<T(\yy,i)$, we have $p_0\leq T(\xx,i)$. We have $T(\xx,n)=p_0$ since $n\in Y$. Because $\Le(\xx,i,n)$ is balanced in $T$ and $T(i,n)=q_0>p_0=T(\xx,n)$, we have $ T(\xx,i)\leq p_0$. Hence, $T(\xx,i)=p_0<T(\yy,i)$. It follows that $i\in Y$. \cref{lem:induction} tells us that $T(\yy,i)=q_0$, which is a contradiction because we have seen that $T(\yy,i)>q_0$. 
\end{proof} 

\begin{figure}[hbpt]
  \begin{center}
  \includegraphics[height=6.114cm]{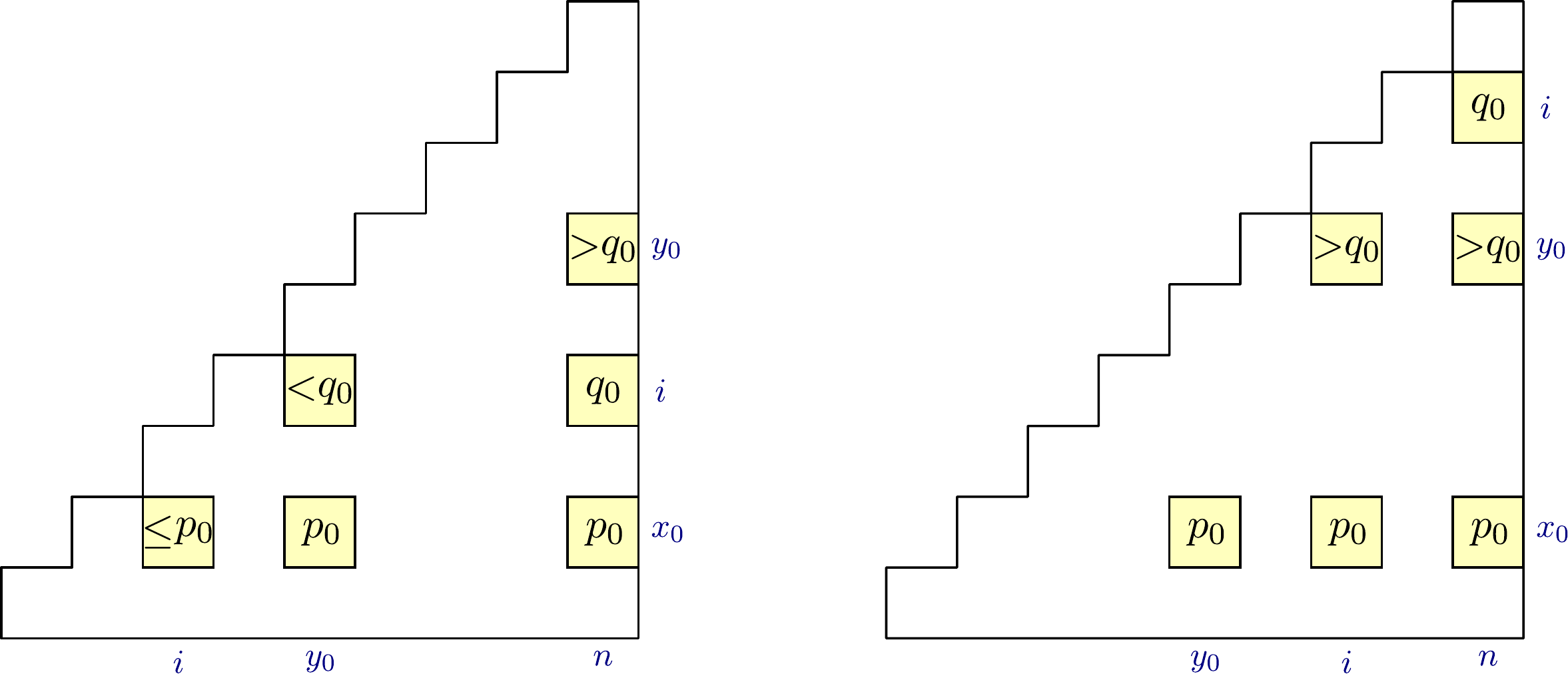}
  \end{center}
\caption{A schematic illustration of some entries in the tableau $T$ in the proof of \cref{lem:ifninY3}. On the left is the case where $i<\yy$; on the right is the case where $i>\yy$. }\label{fig:A3} 
\end{figure}

\cref{lem:ifninY1,lem:ifninY3} imply that \eqref{II'} holds. The next two lemmas are devoted to proving \eqref{IV'}. 

\begin{lemma}\label{lem:ifninY4}
Suppose $n\in Y$. Then every $\Le$-shape intersecting column $n$ is balanced in $\uparrow_{\mathcal B}\! T$. 
\end{lemma} 

\begin{proof}
Let $T^\dd=\uparrow_{\mathcal B}\!T$.
It follows from \cref{lem:induction,lem:ifninY1,lem:ifninY3} that the set of boxes where $T$ and $T^\dd$ disagree is $\mathcal B\cup\{(\yy,y):y\in Y\}$. Moreover, we have $T(\xx,\yy)=p_0$ and $T^\dd(\xx,\yy)=q_0$, and for every $y\in Y$, we have $T(\xx,y)=T^\dd(\yy,y)=p_0$ and $T(\yy,y)=T^\dd(\xx,y)=q_0$.  
Thus, let us fix $1\leq i<j<n$; we will show that $\Le(i,j,n)$ is balanced in $T^\dd$. If the sets $\{i,j\}$ and $\{\xx,\yy\}$ are disjoint, then this is immediate because $\Le(i,j,n)$ is balanced in $T$. Therefore, we may assume $\{i,j\}\cap\{\xx,\yy\}\neq\emptyset$. 

\medskip 

\noindent {\bf Case 1.} Suppose $i\not\in\{\xx,\yy\}$. Then $j\in\{\xx,\yy\}$. When we transform $T$ into $T^\dd$, the entries  in the two lower boxes in $\Le(i,j,n)$ do not change, while the entry in the upper box changes either from $p_0$ to $q_0$ (if $j=\xx$) or from $q_0$ to $p_0$ (if $j=\yy$). If $T(i,n)\leq p_0$ or $T(i,n)\geq q_0$, then $\Le(i,j,n)$ remains balanced when we transform $T$ into $T^\dd$. If $T(i,n)$ is some integer $t\in[p_0+1,q_0-1]$, then by \cref{lem:h_t}, we have $i=h_t$ and $T(i,n)=T(i,\xx)=t$. In this case, $T^\dd(i,n)=T^\dd(i,j)=t$, so $\Le(i,j,n)$ is balanced in $T^\dd$. 

\medskip 

\noindent {\bf Case 2.} Suppose $i\in\{\xx,\yy\}$ and $j\in\{\yy\}\cup Y$. Since $j$ and $n$ are both in $\{\yy\}\cup Y$, we have $T^\dd(i,j)=T^\dd(i,n)$, so $\Le(i,j,n)$ is balanced in $T^\dd$. 

\medskip 

\noindent {\bf Case 3.} Suppose $i\in\{\xx,\yy\}$ and $j\not\in\{\yy\}\cup Y$. We have $T(i,j)=T^\dd(i,j)$ and $T(j,n)=T^\dd(j,n)$. Moreover, $\{T(i,n),T^*(i,n)\}=\{p_0,q_0\}$. \cref{lem:h_t} tells us that all of the integers in the interval $[p_0+1,q_0-1]$ appear below the box $(\xx,n)$ in $T$, so $T(j,n)$ does not belong to that interval. \cref{lem:notinrow} tells us that $T(i,j)$ is not in $[p_0+1,q_0-1]$ either. Hence, $\Le(i,j,n)$ remains balanced when we transform $T$ into $T^\dd$. 
\end{proof} 

\begin{lemma}\label{lem:ifnnotinY}
Suppose $n\not\in Y$. Then every $\Le$-shape intersecting column $n$ is balanced in $\uparrow_{\mathcal B}\! T$. 
\end{lemma} 

\begin{proof}
Let $T^\dd=\uparrow_{\mathcal B}\!T$. It follows from \cref{lem:induction} that the set of boxes where $T$ and $T^\dd$ disagree is $\mathcal B\cup\{(\yy,y):y\in Y\}$. Moreover, we have $T(\xx,\yy)=p_0$ and $T^\dd(\xx,\yy)=q_0$, and for every $y\in Y$, we have $T(\xx,y)=T^\dd(\yy,y)=p_0$ and $T(\yy,y)=T^\dd(\xx,y)=q_0$.  
Thus, let us fix $1\leq i<j<n$; we will show that $\Le(i,j,n)$ is balanced in $T^\dd$. If $i\not\in\{\xx,\yy\}$ or $j\not\in \{\yy\}\cup Y$, then this is immediate because $\Le(i,j,n)$ is balanced in $T$. Therefore, we may assume $i\in\{\xx,\yy\}$ and $j\in\{\yy\}\cup Y$. This implies that $\{T(i,j),T^\dd(i,j)\}=\{p_0,q_0\}$. We also know by \cref{lem:notinrow} that $T(i,n)\not\in[p_0+1,q_0-1]$. 

\medskip 

\noindent {\bf Case 1.} Suppose $T(i,n)<p_0$. Since $\Le(i,j,n)$ is balanced in $T$ and $T(i,j)\in\{p_0,q_0\}$, we have $T(j,n)<T(i,n)$. Then $T^\dd(j,n)<T^\dd(i,n)<T^\dd(i,j)$, so $\Le(i,j,n)$ is balanced in $T^\dd$. 

\medskip 

\noindent {\bf Case 2.} Suppose $T(i,n)>q_0$. Since $\Le(i,j,n)$ is balanced in $T$ and $T(i,j)\in\{p_0,q_0\}$, we have $T(j,n)>T(i,n)$. Then $T^\dd(j,n)>T^\dd(i,n)>T^\dd(i,j)$, so $\Le(i,j,n)$ is balanced in~$T^\dd$. 

\medskip 

\noindent {\bf Case 3.} Suppose $i=\xx$ and $T(i,n)=q_0$. Then $T^\dd(i,j)=T^\dd(i,n)=q_0$, so $\Le(i,j,n)$ is balanced in~$T^\dd$. 

\medskip 

\noindent {\bf Case 4.} Suppose $i=\yy$ and $T(i,n)=p_0$. Then $T^\dd(i,j)=T^\dd(i,n)=p_0$, so $\Le(i,j,n)$ is balanced in~$T^\dd$. 

\medskip 

\noindent {\bf Case 5.} Suppose $i=\xx$ and $T(i,n)=p_0$. Because $n\not\in Y$ and $T(\xx,n)=p_0$, we must have $T(\yy,n)<p_0$. Thus, $T(\yy,j)=q_0>T(\yy,n)$. Since $\Le(\yy,j,n)$ is balanced in $T$, we must have $T(j,n)<T(\yy,n)<p_0$. It follows that $T^\dd(j,n)<p_0=T^\dd(i,n)<q_0=T^\dd(i,j)$, so $\Le(i,j,n)$ is balanced in $T^\dd$. 

\medskip 

\noindent {\bf Case 6.} Suppose $i=\yy$ and $T(i,n)=q_0$. Then $T(\xx,\yy)=p_0$ and $T(\xx,n)\neq T(\yy,n)=q_0$. Since $\Le(\xx,\yy,n)$ is balanced in $T$, we must have $p_0\leq T(\xx,n)<q_0$. We know by \cref{lem:notinrow} that $T(\xx,n)\not\in[p_0+1,q_0-1]$, so $T(\xx,n)=p_0$. This shows that $T(\xx,n)=p_0<T(\yy,n)$, which contradicts the hypothesis that $n\not\in Y$. Hence, this case cannot happen. 
\end{proof} 

We have now proven \eqref{II'}, \eqref{III'}, \eqref{IV'}. This completes the proof of \cref{prop:construct}. 

\subsection{Proving $\Phi$ is an Isomorphism} 

We now have the tools needed to prove that the map $\Phi\colon\PD(w)\to\LT(w)$ is a poset isomorphism. Suppose $T,T'\in\IT(w)$ are inversions tableaux for a permutation $w\in\SSS_n$. Assume that $\Lambda(T)<\Lambda(T')$; equivalently, $T'=\uparrow_{\mathcal M}\!T$ for some nonempty multiset $\mathcal M$ of boxes. Let $P=\Theta^{-1}(T)$ and $P'=\Theta^{-1}(T')$. Our goal is to show that we can obtain $P'$ from $P$ via a sequence of chute moves. Proceeding by induction on the size of $\mathcal M$, we will see that it suffices to find a set $\mathcal B\subseteq \mathcal M$ such that $\uparrow_{\mathcal B}\!T$ is an inversions tableau and such that $\Theta^{-1}(\uparrow_{\mathcal B}\!T)=\C_{i,j}(P)$ for some box $(i,j)$. \cref{prop:construct} provides a way to do this (by \cref{prop:chute}). However, \cref{prop:construct} has a crucial hypothesis that the set $\mathcal X$ (defined in the statement of that proposition) is nonempty. Our objective is to show that $\mathcal X$ is indeed nonempty.   

\begin{lemma}\label{lem:isomorphism1}
Let $w\in\SSS_n$, and let $T,T'\in\IT(w)$ be such that $T'=\uparrow_{\mathcal M}\!T$, where $\mathcal M$ is a multiset consisting of exactly one box of the form $(\xx,n)$ with some multiplicity $m\geq 1$. Then $\uparrow_{(\xx,n)}\!T$ is an inversions tableau. 
\end{lemma}
\begin{proof}
Let $T^\dd=\uparrow_{(\xx,n)}\!T$. We know by \cref{lem:bigger_after} that $T(i,j)\leq T'(i,j)$ for all $(i,j)\in\LRT_{n-1}$. Therefore, as we increment the box $(\xx,n)$ $m$ times to transform $T$ into $T'$, no box can ever have its entry decrease (otherwise, it would not be able to increase again). It follows that all of these increments are pure increments. Moreover, $T(i,j)=T^\dd(i,j)=T'(i,j)$ for every box $(i,j)$ other than $(\xx,n)$. Each $\Le$-shape that does not contain $(\xx,n)$ has the same entries in $T^\dd$ as in $T$, so it is balanced in $T^\dd$. 

Consider a $\Le$-shape of the form $\Le(r,\xx,n)$. This $\Le$-shape is balanced in $T$ and $T'$. Therefore, $T(r,n)$ is weakly between $T(r,\xx)$ and $T(\xx,n)$, and it is also weakly between $T'(r,\xx)$ and $T'(\xx,n)$. Since $T(\xx,n)<T^\dd(\xx,n)\leq T'(\xx,n)$, we deduce that $\Le(r,\xx,n)$ is balanced in $T^\dd$. 

Finally, consider a $\Le$-shape of the form $\Le(\xx,s,n)$. This $\Le$-shape is balanced in $T$ and $T'$, so $T(\xx,n)$ and $T'(\xx,n)$ are both weakly between the number $T(\xx,s)=T^\dd(\xx,s)=T'(\xx,s)$ and the number $T(s,n)=T^\dd(s,n)=T'(s,n)$. Since $T(\xx,n)<T^\dd(\xx,n)\leq T'(\xx,n)$, we deduce that $\Le(\xx,s,n)$ is balanced in $T^\dd$. 
\end{proof} 

\begin{lemma}\label{lem:isomorphism2}
Let $w\in\SSS_n$, and let $T,T'\in\IT(w)$ be such that $T'=\uparrow_{\mathcal M}\!T$, where $\mathcal M$ is a nonempty multiset of boxes contained in a single row. Then $\Theta^{-1}(T)\leqC\Theta^{-1}(T')$. 
\end{lemma}
\begin{proof}
The proof is trivial if $\mathcal M$ is empty, so we may assume $\mathcal M$ is nonempty and proceed by induction on the sum of the multiplicities of the elements of $\mathcal M$. Let $\mathcal X$ be the set of boxes $(i,j)$ in $\mathcal M$ such that $\uparrow_{(i,j)}\!T_{\leq j}$ is an inversions tableau. Let $d$ be the leftmost column containing a box in $\mathcal M$, and let $\mathcal M_d$ be the multiset of boxes obtained from $\mathcal M$ by deleting all boxes not in column $d$. Then $\mathcal M_d$ contains only one box $(i,d)$ (with some multiplicity). Moreover, $T_{\leq d}$ and $T'_{\leq d}$ are inversions tableaux, and $T'_{\leq d}=\uparrow_{\mathcal M_d}\!T_{\leq d}$. It follows from \cref{lem:isomorphism1} that $\uparrow_{(i,d)}\!T_{\leq d}\in\mathcal X$, so $\mathcal X$ is nonempty. Therefore, we can choose $(\xx,\yy)\in\mathcal X$ such that no other box in $\mathcal X$ appears to the right of $(\xx,\yy)$. Let $Y$ and $\mathcal B$ be as in \cref{prop:construct}. We know by \cref{prop:construct} that $\mathcal B\subseteq\mathcal M$ and that $\uparrow_{\mathcal B}\!T$ is an inversions tableau. Let $\mathcal M'$ be the multiset obtained from $\mathcal M$ by decreasing the multiplicity of each element of $\mathcal B$ by $1$. Then $\uparrow_{\mathcal M'}\!(\uparrow_{\mathcal B}\!T)=T'$, so we know by induction that $\Theta^{-1}(\uparrow_{\mathcal B}\!T)\leqC\Theta^{-1}(T')$. It follows from \cref{prop:chute,prop:construct} that $\Theta^{-1}(\uparrow_{\mathcal B}\!T)=\C_{\xx,\yy}(\Theta^{-1}(T))$, so $\Theta^{-1}(T)\leqC\Theta^{-1}(\uparrow_{\mathcal B}\!T)\leqC\Theta^{-1}(T')$.  
\end{proof} 

\begin{lemma}\label{lem:isomorphism3}
Let $w\in\SSS_n$, and let $T,T'\in\IT(w)$ be such that $T'=\uparrow_{\mathcal M}\!T$, where $\mathcal M$ is a nonempty multiset of boxes contained in column $n$. Then there exists a box $(\xx,n)$ in $\mathcal M$ such that $\uparrow_{(\xx,n)}\!T$ is an inversions tableau.  
\end{lemma}
\begin{proof}
Let $P=\Theta^{-1}(T)$ and $P'=\Theta^{-1}(T')$. The Lehmer tableaux $\Phi(P)$ and $\Phi(P')$ agree in all columns other than column $n$, and $\Phi(P)\leq\Phi(P')$. According to \cref{prop:transpose}, the Lehmer tableaux $\Phi(P^\top)$ and $\Phi((P')^\top)$ agree in all rows other than row $w^{-1}(n)$, and $\Phi((P')^\top)\leq\Phi(P^\top)$. We know by \cref{lem:isomorphism2} that $(P')^\top\leqC P^\top$. Since the map sending a reduced pipe dream to its transpose is an anti-isomorphism from $\PD(w)$ to $\PD(w^{-1})$, we deduce that $P\leqC P'$. This implies that there is some pair $(\xx,\yy)$ such that $P\leqC\C_{\xx,\yy}(P)\leqC P'$. According to \cref{prop:chute}, we have $\Theta(\C_{\xx,\yy}(P))=\uparrow_{\mathcal B}\!T$ for some set $\mathcal B$ of boxes that contains $(\xx,\yy)$ and is contained in a single row. \cref{cor:easy_direction} tells us that $\Phi(P)\leq\Phi(\C_{\xx,\yy}(P))\leq\Phi(P')$, so $\Phi(P)$ and $\Phi(\C_{\xx,\yy}(P))$ agree in all columns other than column $n$. This implies that $T$ and $\uparrow_{\mathcal B}\!T$ agree in all columns other than column $n$, so $\mathcal B=\{(\xx,\yy)\}=\{(\xx,n)\}$. Hence, $\uparrow_{(\xx,n)}\!T$ is an inversions tableau.  
\end{proof} 

The proof of the next lemma is essentially the same as that of \cref{lem:isomorphism2}, except now we use \cref{lem:isomorphism3} instead of \cref{lem:isomorphism1}. 

\begin{lemma}\label{lem:isomorphism4}
Let $w\in\SSS_n$, and let $T,T'\in\IT(w)$ be such that $T'=\uparrow_{\mathcal M}\!T$, where $\mathcal M$ is a nonempty multiset of boxes. Then $\Theta^{-1}(T)\leqC\Theta^{-1}(T')$. 
\end{lemma}
\begin{proof}
The proof is trivial if $\mathcal M$ is empty, so we may assume $\mathcal M$ is nonempty and proceed by induction on the sum of the multiplicities of the elements of $\mathcal M$. Let $\mathcal X$ be the set of boxes $(i,j)$ in $\mathcal M$ such that $\uparrow_{(i,j)}\!T_{\leq j}$ is an inversions tableau. Let $d$ be the leftmost column containing a box in $\mathcal M$, and let $\mathcal M_d$ be the multiset of boxes obtained from $\mathcal M$ by deleting all boxes not in column $d$. Then $T_{\leq d}$ and $T'_{\leq d}$ are inversions tableaux, and $T'_{\leq d}=\uparrow_{\mathcal M_d}\!T_{\leq d}$. It follows from \cref{lem:isomorphism3} that $\mathcal X$ is nonempty, so we can choose $(\xx,\yy)\in\mathcal X$ such that no other box in $\mathcal X$ appears to the right of $(\xx,\yy)$. Let $Y$ and $\mathcal B$ be as in \cref{prop:construct}. We know by \cref{prop:construct} that $\mathcal B\subseteq\mathcal M$ and that $\uparrow_{\mathcal B}\!T$ is an inversions tableau. Let $\mathcal M'$ be the multiset obtained from $\mathcal M$ by decreasing the multiplicity of each element of $\mathcal B$ by $1$. Then $\uparrow_{\mathcal M'}\!(\uparrow_{\mathcal B}\!T)=T'$, so we know by induction that $\Theta^{-1}(\uparrow_{\mathcal B}\!T)\leqC\Theta^{-1}(T')$. It follows from \cref{prop:chute,prop:construct} that $\Theta^{-1}(\uparrow_{\mathcal B}\!T)=\C_{\xx,\yy}(\Theta^{-1}(T))$, so $\Theta^{-1}(T)\leqC\Theta^{-1}(\uparrow_{\mathcal B}\!T\leqC\Theta^{-1}(T'))$.  
\end{proof} 

\begin{theorem}\label{thm:isomorphism} 
Let $w\in\SSS_n$. The map $\Phi\colon\PD(w)\to\LT(w)$ is a poset isomorphism. 
\end{theorem} 
\begin{proof}
We know already that $\Phi$ is a bijection. Let $P,P'\in\PD(w)$. We know by \cref{cor:easy_direction} that if $P\leqC P'$, then $\Phi(P)\leq\Phi(P')$. 

Now suppose $\Phi(P)\leq \Phi(P')$. Let $T=\Theta(P)$ and $T'=\Theta(P')$. Then $T'=\uparrow_{\mathcal M}\!T$ for some multiset $\mathcal M$ of boxes, so \cref{lem:isomorphism4} tells us that $P\leqC P'$. 
\end{proof} 

\section{The Lattice Property}\label{sec:lattice} 
We will now complete the proof of \cref{thm:main}. To do so, we will use the following result of Bj\"orner, Edelman, and Ziegler. 

\begin{proposition}[{\cite{BEZ}}]\label{prop:BEZ}
Let $Q$ be a finite poset with a minimum element and a maximum
element. Suppose that for all distinct $\gamma_0,\gamma_1,\gamma_2\in Q$ satisfying $\gamma_0\lessdot\gamma_1$ and $\gamma_0\lessdot\gamma_2$, the elements
$\gamma_1$ and $\gamma_2$ have a least upper bound in $Q$. Then $Q$ is a lattice.
\end{proposition} 

Fix $w\in\SSS_n$. Rubey proved that $\PD(w)$ has a minimum element and a maximum element \cite[Theorem~3.8]{Rubey}. To apply \cref{prop:BEZ}, let us fix distinct pipe dreams ${P_0,P_1,P_2\in\PD(w)}$ such that $P_0\lessdot_{\C} P_1$ and $P_0\lessdot_{\C} P_2$ (if no such pipe dreams exist, then $\PD(w)$ vacuously satisfies the hypotheses of \cref{prop:BEZ}). For $\epsilon\in\{0,1,2\}$, let $T_\epsilon=\Theta(P_\epsilon)$ and $L_\epsilon=\Phi(P_\epsilon)=\Lambda(T_\epsilon)$. For each $\delta\in\{1,2\}$, there is an inversion $(i_\delta,j_\delta)$ of $w$ such that $P_\delta$ is obtained by applying the chute move $\C_{i_\delta,j_\delta}$ to $P_0$; that is, $P_\delta=\C_{i_\delta,j_\delta}(P_0)$. Let $R_\delta=\RR(P_0,P_\delta)$ be the rectangle where this chute move takes place. The corner boxes $\NW(R_\delta)$ and $\SW(R_\delta)$ are filled with bump tiles, the box $\SE(R_\delta)$ is filled with a bump or elbow tile, and all other boxes in $R_\delta$ are filled with cross tiles. Moreover, $P_\delta$ is obtained from $P_0$ by swapping the tiles in $\SW(R_\delta)$ and $\NE(R_\delta)$. 

The next lemma is illustrated in \cref{fig:commute}. 

\begin{figure}[ht]
  \begin{center}
  \includegraphics[width=\linewidth]{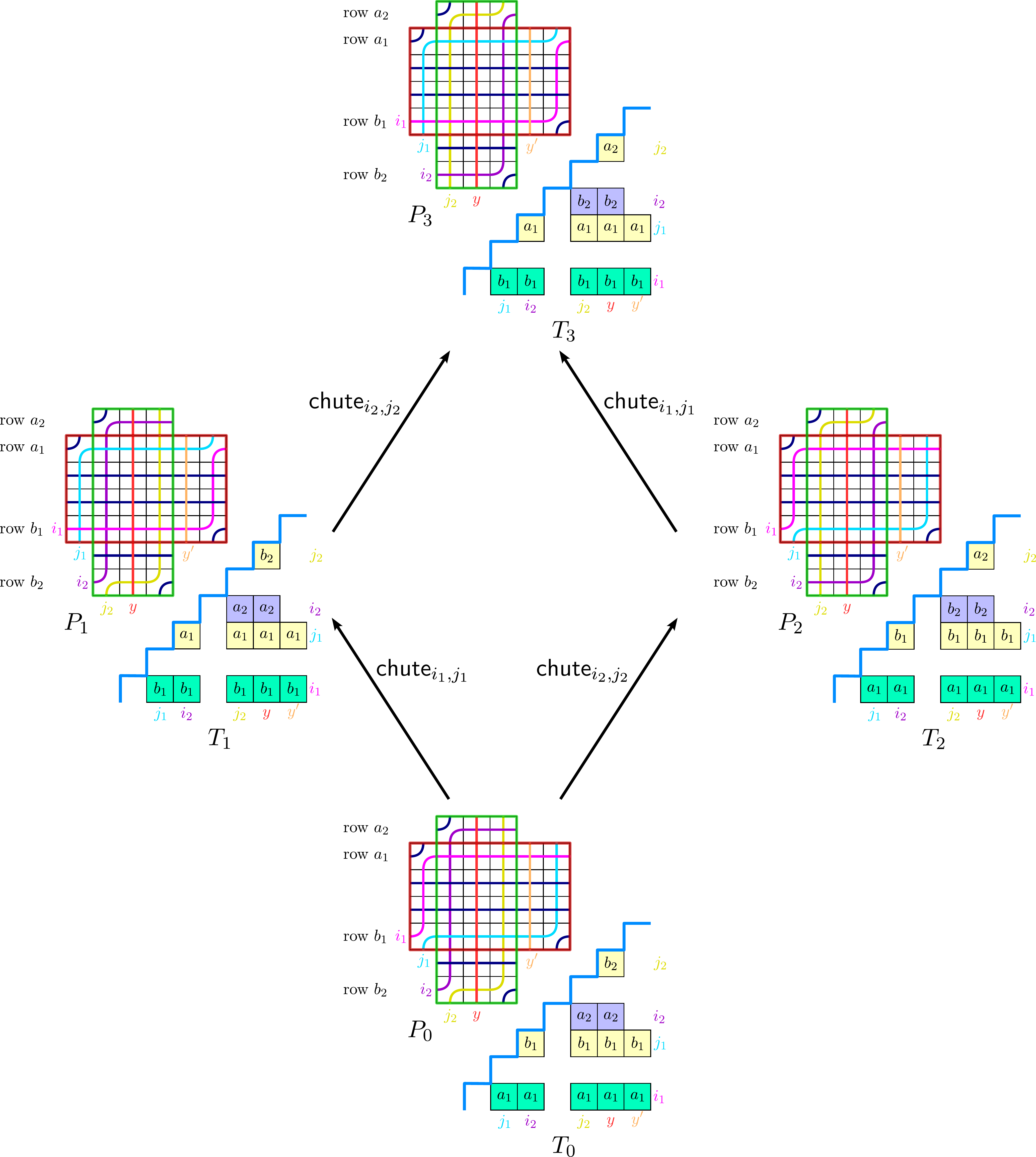}
  \end{center}
\caption{An illustration of \cref{lem:commute}. The rectangles $R_1$ and $R_2$ are outlined in {\color{MildRed}red} and {\color{MildGreen}green}, respectively. The sets $\mathcal B_1$ and $\mathcal B_2$ are colored in {\color{Turquoise}turquoise} and {\color{Periwinkle}periwinkle}, respectively. }\label{fig:commute} 
\end{figure}

\begin{lemma}\label{lem:commute} 
Suppose that $\SW(R_1)$ is not a corner of $R_2$ and that $\SW(R_2)$ is not a corner of $R_1$. Then $P_1$ and $P_2$ have a least upper bound $P_1\vee P_2$ in the chute move poset $\PD(w)$, and $\Phi(P_1\vee P_2)$ is the componentwise maximum of $\Phi(P_1)$ and $\Phi(P_2)$. Moreover, the interval $[P_0,P_1\vee P_2]$ is a diamond. 
\end{lemma}
\begin{proof}
It is straightforward to check that the chute moves $\C_{i_1,j_1}$ and $\C_{i_2,j_2}$ commute when we apply them to $P_0$. More precisely, it is possible to apply $\C_{i_1,j_1}$ to $P_2$ and to apply $\C_{i_2,j_2}$ to $P_1$, and we have $\C_{i_1,j_1}(P_2)=\C_{i_2,j_2}(P_1)$. Let $P_3=\C_{i_1,j_1}(P_2)$, and let $T_3=\Theta(P_3)$ and ${L_3=\Phi(P_3)=\Lambda(T_3)}$. For each $\delta\in\{1,2\}$, let $Y_\delta$ be the set of indices of pipes that cross vertically through $R_\delta$ in $P_0$.
Let ${\mathcal B_\delta=\{(i_\delta,j_\delta)\}\cup\{(i_\delta,y):y\in Y_\delta\}}$. It follows from \cref{prop:chute} that $T_\delta=\uparrow_{\mathcal B_\delta}\!T_0$. Upon inspection, we find that $Y_\delta$ is also the set of indices of pipes that cross vertically through $R_{3-\delta}$ in $P_{3-\delta}$. Hence, $T_3=\uparrow_{\mathcal B_1}\!T_2=\uparrow_{\mathcal B_2}\!T_1$. 

For each $\delta\in\{1,2\}$, we obtain the Lehmer tableau $L_\delta$ from $L_0$ by increasing the entries in the boxes in $\mathcal B_\delta$ by $1$ each. The sets $\mathcal B_1$ and $\mathcal B_2$ are disjoint, so we obtain $L_3$ from $L_0$ by increasing the entries in the boxes in $\mathcal B_1\cup\mathcal B_2$ by $1$ each. It follows that $L_3$ is the componentwise maximum of $L_1$ and $L_2$, so $L_3$ must be the least upper bound of $L_1$ and $L_2$ in $\LT(w)$. It follows from \cref{thm:isomorphism} that $P_3$ is the least upper bound of $P_1$ and $P_2$ in $\PD(w)$. Moreover, $[P_0,P_3]$ is a diamond. 
\end{proof}

Now that we have established \cref{lem:commute}, we may henceforth assume that the southwest corner of one of $R_1$ and $R_2$ is a corner of the other. Without loss of generality, let us assume that $\SW(R_1)$ is a corner of $R_2$. A moment's consideration reveals that $\SW(R_1)$ cannot be $\SW(R_2)$ or $\NE(R_2)$, so it must be $\NW(R_2)$ or $\SE(R_2)$. Somewhat surprisingly, the case where $\SW(R_1)=\NW(R_2)$ and the case where $\SW(R_1)=\SE(R_2)$ require fundamentally different arguments. 

\begin{figure}[]
  \begin{center}
  \includegraphics[height=20.42cm]{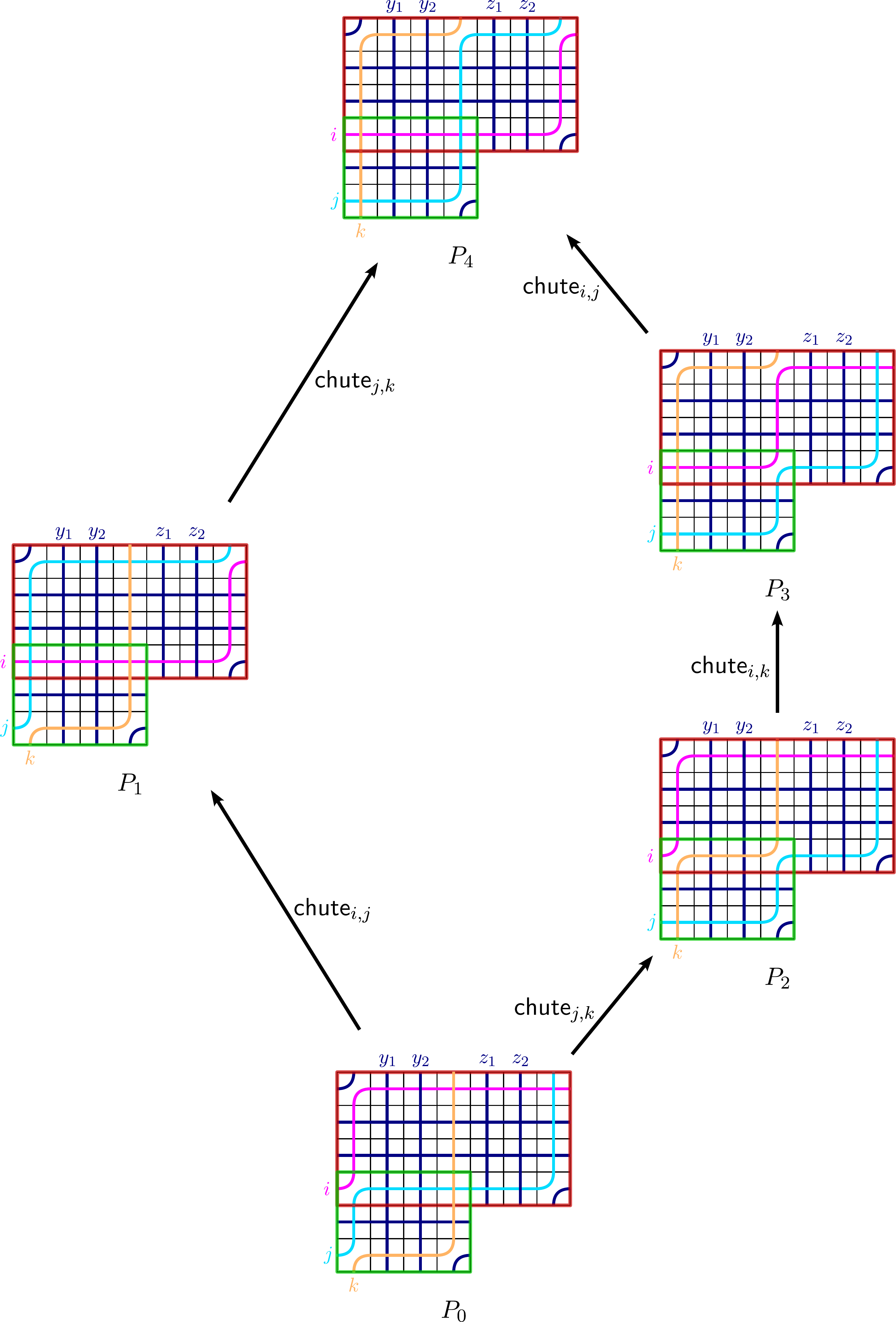}
  \end{center}
\caption{A schematic illustration of a pentagon formed by pipe dreams $P_0,P_1,P_2,P_3,P_4$ when the southwest corner of $R_1$ (outlined in {\color{MildRed}red}) is the northwest corner of $R_2$ (outlined in {\color{MildGreen}green}).  }\label{fig:pipes_pentagon} 
\end{figure} 

\begin{lemma}\label{lem:non-commutative1}
Suppose that $\SW(R_1)=\NW(R_2)$. Then $P_1$ and $P_2$ have a least upper bound $P_1\vee P_2$ in the chute move poset $\PD(w)$, and $\Phi(P_1\vee P_2)$ is the componentwise maximum of $\Phi(P_1)$ and $\Phi(P_2)$. Moreover, the interval $[P_0,P_1\vee P_2]$ is a pentagon. 
\end{lemma} 
\begin{proof}
Recall that $P_1=\C_{i_1,j_1}(P_0)$ and $P_2=\C_{i_2,j_2}(P_0)$. We have $j_1=i_2$. Let $i=i_1$, $j=j_1=i_2$, and $k=j_2$. Then $i<j<k$. \cref{fig:pipes_pentagon} provides a schematic illustration of the relevant parts of $P_0$, $P_1$, and $P_2$, along with two other pipe dreams $P_3=\C_{i,k}(P_2)$ and ${P_4=\C_{j,k}(P_1)=\C_{i,j}(P_3)}$. We will show that $P_4$ is the least upper bound of $P_1$ and $P_2$. 

Let $a,b,c$ be the indices of the rows of $\ULT_n$ containing the boxes $\NW(R_1),\SW(R_1),\SW(R_2)$, respectively. For each $\epsilon\in\{0,1,2,3,4\}$, let $T_\epsilon=\Theta(P_\epsilon)$, and let $L_\epsilon=\Phi(P_\epsilon)=\Lambda(T_\epsilon)$. \cref{fig:tableaux_pentagon} shows schematic illustrations of the relevant parts of these inversions tableaux. Let $y_1,\ldots,y_\ell$ be the indices of pipes that cross vertically through $R_2$ in $P_0$; note that these pipes also cross vertically through $R_1$ in $P_0$. The pipe $k$ also crosses vertically through $R_1$ in $P_0$. Let 
$z_1,\ldots,z_m$ be the indices of pipes that cross vertically through $R_1$ in $P_0$ but do not belong to the set $\{y_1,\ldots,y_\ell,k\}$. 
For each pair $(\epsilon,\epsilon')\in\{(0,1),(0,2),(1,4),(2,3),(3,4)\}$, we know by \cref{prop:chute} that there is a set $\mathcal B_{\epsilon,\epsilon'}$ of boxes such that $T_{\epsilon'}=\uparrow_{\mathcal B_{\epsilon,\epsilon'}}\!T_\epsilon$. Moreover, by \cref{prop:chute}, we have  
\begin{align*}
\mathcal B_{0,1}&=\{(i,j),(i,y_1),\ldots,(i,y_\ell),(i,k),(i,z_1),\ldots,(i,z_m)\}, \\ 
\mathcal B_{0,2}&=\mathcal B_{1,4}=\{(j,k),(j,y_1),\ldots,(j,y_\ell)\}, \\ 
\mathcal B_{2,3}&=\{(i,k),(i,y_1),\ldots,(i,y_\ell)\}, \\ 
\mathcal B_{3,4}&=\{(i,j),(i,z_1),\ldots,(i,z_m)\}.
\end{align*} 
We obtain $L_{\epsilon'}$ from $L_\epsilon$ by increasing the entries in the boxes in $\mathcal B_{\epsilon,\epsilon'}$ by $1$ each. Because $\mathcal B_{0,1}$ and $\mathcal B_{0,2}$ are disjoint and $\mathcal B_{0,2}=\mathcal B_{1,4}$, the tableau $L_4$ is the componentwise maximum of $L_1$ and $L_2$. This implies that $L_4$ is the least upper bound of $L_1$ and $L_2$ in $\LT(w)$. By \cref{thm:isomorphism}, the pipe dream $P_4$ is the least upper bound of $P_1$ and $P_2$ in $\PD(w)$. Moreover, the interval $[P_0,P_4]$ is a pentagon. 
\end{proof} 

\begin{figure}[]
  \begin{center}
  \includegraphics[height=21.36cm]{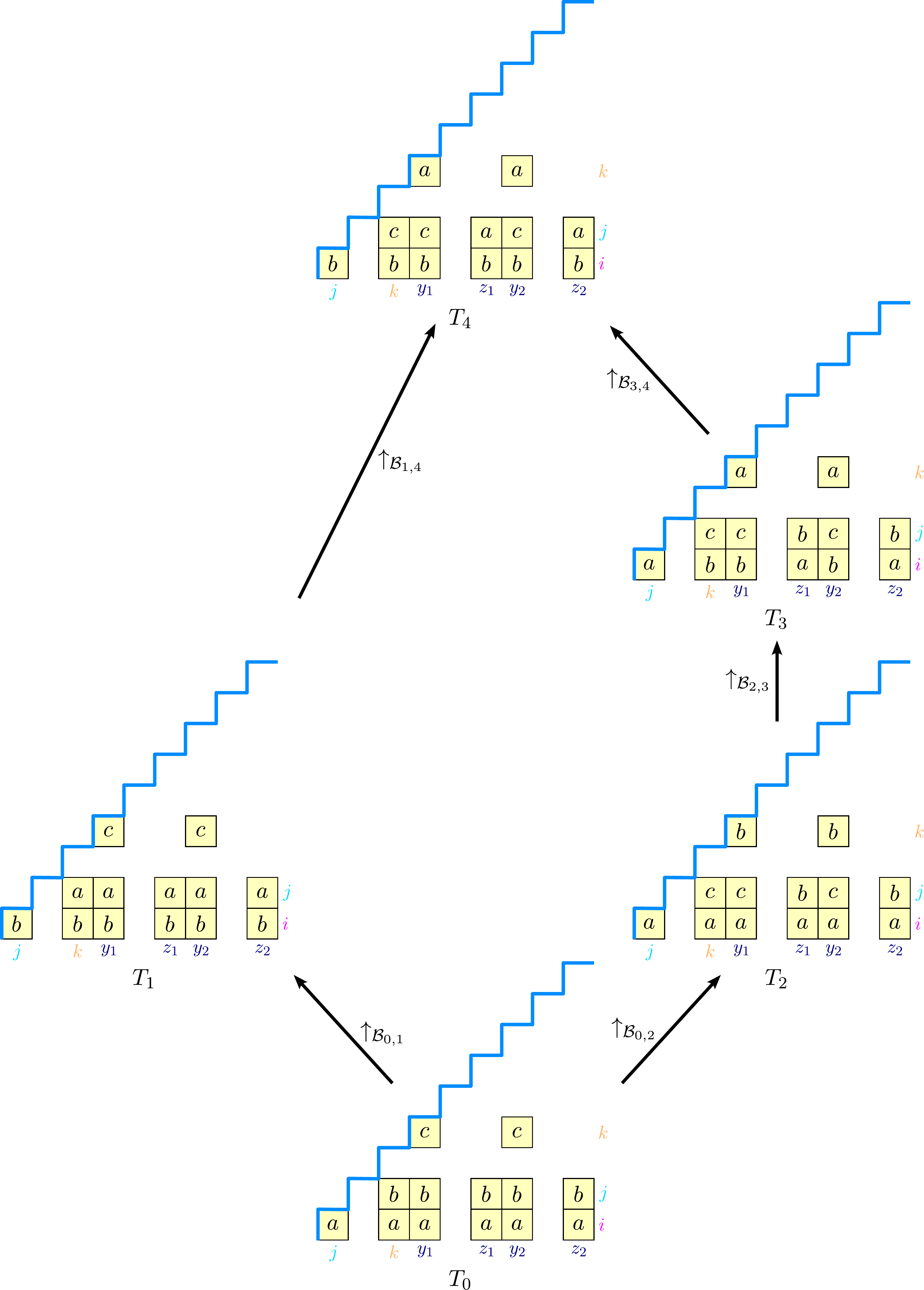}
  \end{center}
\caption{A schematic illustration of the inversions tableaux $T_0,T_1,T_2,T_3,T_4$ corresponding to the pipe dreams in \cref{fig:pipes_pentagon}.  }\label{fig:tableaux_pentagon} 
\end{figure}

If $\SW(R_1)=\NW(R_2)$, then we saw in the proof of \cref{lem:non-commutative1} that the componentwise maximum of $\Phi(P_1)$ and $\Phi(P_2)$ is the Lehmer tableau of the least upper bound of $P_1$ and $P_2$. This is no longer the case if $\SW(R_1)=\SE(R_2)$. In this case, we will appeal to the triforce embedding from \cref{subsec:Triforce}. 

\begin{lemma}\label{lem:non-commutative2}
Suppose that $\SW(R_1)=\SE(R_2)$. Then $P_1$ and $P_2$ have a least upper bound $P_1\vee P_2$ in the chute move poset $\PD(w)$. Moreover, the interval $[P_0,P_1\vee P_2]$ is a pentagon. 
\end{lemma} 
\begin{proof}
For each $\delta\in\{1,2\}$, it follows from \cref{prop:Triforce} that $P_\delta^{\Tri}$ covers $P_0^{\Tri}$ in $\PD(w^{\Tri})$. In other words, $P_\delta^{\Tri}$ is obtained by applying a chute move to $P_0^{\Tri}$; let $R_\delta^{\Tri}=\RR(P_0^{\Tri},P_\delta^{\Tri})$ be the rectangle in $\ULT_{2n}$ where this chute move takes place. We obtain $R_\delta^{\Tri}$ by reflecting $R_\delta$ through a line of slope~$1$. Since $\SW(R_1)=\SE(R_2)$, it follows that $\SW(R_1^{\Tri})=\NW(R_2^{\Tri})$. We deduce from \cref{lem:non-commutative1} (applied to $w^{\Tri}$ instead of $w$) that $P_1^{\Tri}$ and $P_2^{\Tri}$ have a least upper bound $P_1^{\Tri}\vee P_2^{\Tri}$ in $\PD(w^{\Tri})$. Moreover, the interval $[P_0^{\Tri},P_1^{\Tri}\vee P_2^{\Tri}]$ is a pentagon. Invoking \cref{prop:Triforce} again, we find that $P_1$ and $P_2$ have a least upper bound $P_1\vee P_2$ in $\PD(w)$ and that the interval $[P_0,P_1\vee P_2]$ is a pentagon. 
\end{proof} 

\begin{remark}
It is possible to prove \cref{lem:non-commutative2} via an argument similar to the one used to prove \cref{lem:non-commutative1}. However, the proof in this case is more complicated and requires repeated applications of \cref{lem:less_hook}. We have chosen to instead use the triforce embedding because it simplifies the argument and foreshadows our proof that $\PD(w)$ is semidistributive.  
\end{remark} 

\cref{lem:commute,lem:non-commutative1,lem:non-commutative2} allow us to apply \cref{prop:BEZ} to complete the proof of \cref{thm:main}. 

We also have the following corollary, which proves part of \cref{thm:semidistributive}. 

\begin{corollary}\label{cor:polygonal} 
For every permutation $w\in\SSS_n$, the lattice $\PD(w)$ is polygonal, and each of its polygons is a diamond or a pentagon.  
\end{corollary} 
\begin{proof}
If $w\in\SSS_n$ and $P,P',P''\in\PD(w)$ are such that $P\lessdot_{\C}P'$ and $P\lessdot_{\C}P''$, then it follows from \cref{lem:commute,lem:non-commutative1,lem:non-commutative2} that the interval $[P,P'\vee P'']$ is a diamond or a pentagon. 

If $w\in\SSS_n$ and $P,P',P''\in\PD(w)$ are such that $P'\lessdot_{\C}P$ and $P''\lessdot_{\C}P$, then the interval $[P'\wedge P'',P]$ is a diamond or a pentagon; indeed, this follows from the preceding paragraph and the fact that $\PD(w)$ is anti-isomorphic to $\PD(w^{-1})$. 
\end{proof} 

\section{Semidistributivity}\label{sec:semidistributivity}  

Let $w\in\SSS_n$. Now that we have established that $\PD(w)$ is a lattice, we can prove that it is semidistributive. 

\begin{lemma}\label{lem:semi1}
Suppose $P_{\dd},P_{\dd\dd}\in\PD(w)$ are such that $P_\dd\lessdot_{\C}P_{\dd\dd}$. Let \[{X=\{P\in\PD(w):P\wedge P_{\dd\dd}=P_{\dd}\}}.\] Suppose $P_0,P_1,P_2\in X$ are such that $P_1=\C_{i_1,j_1}(P_0)$ and $P_2=\C_{i_2,j_2}(P_0)$ for some distinct inversions $(i_1,j_1)$ and $(i_2,j_2)$ of $w$. Let $R_1=\RR(P_0,P_1)$ and $R_2=\RR(P_0,P_2)$. Suppose that $\SW(R_1)\neq\SE(R_2)$ and $\SW(R_2)\neq\SE(R_1)$. If $i_1\neq i_2$, then $P_1\vee P_2\in X$. 
\end{lemma} 
\begin{proof}
Let $(i_\dd,j_\dd)$ be the inversion of $w$ such that $P_{\dd\dd}=\C_{i_\dd,j_\dd}(P_\dd)$. It follows from \cref{prop:chute} that there is a set $\mathcal B_\dd$ of boxes in row $i_\dd$ of $\LRT_{n-1}$ such that $\Phi(P_{\dd\dd})$ is obtained from $\Phi(P_\dd)$ by increasing the entries in $\mathcal B_\dd$ by $1$ each. Suppose $i_1\neq i_2$, and choose $\delta\in\{1,2\}$ such that $i_\delta\neq i_\dd$. It follows from \cref{lem:commute,lem:non-commutative1} that there is a set $\mathcal B_\delta$ of boxes in row $i_\delta$ such that $\Phi(P_1\vee P_2)$ is obtained from $\Phi(P_{3-\delta})$ by increasing the entries in $\mathcal B_\delta$ by $1$ each. Since $P_{3-\delta}\in X$, we have $\Phi(P_\dd)\leq\Phi(P_{3-\delta})$ and $\Phi(P_{\dd\dd})\not\leq\Phi(P_{3-\delta})$. Because $i_\delta\neq i_\dd$, it follows that $\Phi(P_{\dd\dd})\not\leq\Phi(P_1\vee P_2)$, so $P_{\dd\dd}\not\leqC P_1\vee P_2$. But $P_\dd\leqC P_1\leqC P_1\vee P_2$, so $P_1\vee P_2\in X$. 
\end{proof}

\begin{lemma}\label{lem:semi2}
Suppose $P_{\dd},P_{\dd\dd}\in\PD(w)$ are such that $P_\dd\lessdot_{\C}P_{\dd\dd}$. Let \[{X=\{P\in\PD(w):P\wedge P_{\dd\dd}=P_{\dd}\}}.\] Suppose $P_0,P_1,P_2\in X$ are such that $P_1=\C_{i_1,j_1}(P_0)$ and $P_2=\C_{i_2,j_2}(P_0)$ for some distinct inversions $(i_1,j_1)$ and $(i_2,j_2)$ of $w$. Let $R_1=\RR(P_0,P_1)$ and $R_2=\RR(P_0,P_2)$. Suppose that $\SW(R_1)\neq\NW(R_2)$ and $\SW(R_2)\neq\NW(R_1)$. If $j_1\neq j_2$, then $P_1\vee P_2\in X$. 
\end{lemma} 
\begin{proof}
It follows from \cref{prop:Triforce} that $P_\dd^{\Tri}\lessdot_{\C}P_{\dd\dd}^{\Tri}$. Let \[{X^{\Tri}=\{P\in\PD(w^{\Tri}):P\wedge P_{\dd\dd}^{\Tri}=P_{\dd}^{\Tri}\}}.\] Then $X^{\Tri}=\{P^{\Tri}:P\in X\}$, so $P_0^{\Tri},P_1^{\Tri},P_2^{\Tri}\in X$. It follows from the definition of the triforce embedding that $P_1^{\Tri}=\C_{2n+1-j_1,2n+1-i_1}(P_0^{\Tri})$ and $P_2^{\Tri}=\C_{2n+1-j_2,2n+1-i_2}(P_0^{\Tri})$. For each $\delta\in\{1,2\}$, let $R_\delta^{\Tri}=\RR(P_0^{\Tri},P_\delta^{\Tri})$. There is a line $\mathfrak l$ of slope $1$ such that $R_1^{\Tri}$ and $R_2^{\Tri}$ are obtained by reflecting $R_1$ and $R_2$, respectively, through $\mathfrak l$. Therefore, the hypothesis that $\SW(R_1)\neq\NW(R_2)$ and $\SW(R_2)\neq\NW(R_1)$ implies that $\SW(R_1^{\Tri})\neq\SE(R_2^{\Tri})$ and $\SW(R_2^{\Tri})\neq\SE(R_1^{\Tri})$. 

Suppose $j_1\neq j_2$. Then $2n+1-j_1\neq 2n+1-j_2$, so we can apply \cref{lem:semi1} to find that $P_1^{\Tri}\vee P_2^{\Tri}\in X^{\Tri}$. Appealing to \cref{prop:Triforce}, we find that $P_1^{\Tri}\vee P_2^{\Tri}=(P_1\vee P_2)^{\Tri}$, which implies that $P_1\vee P_2\in X$. 
\end{proof} 

\begin{proof}[Proof of \cref{thm:semidistributive}] 
In light of \cref{cor:polygonal}, we just need to prove that every chute move poset is semidistributive. Because every chute move poset is anti-isomorphic to another chute move poset ($\PD(w)$ is anti-isomorphic to $\PD(w^{-1})$), it suffices to prove that every chute move poset is meet-semidistributive. 

Fix a permutation $w\in\SSS_n$. Suppose $P_{\dd},P_{\dd\dd}\in\PD(w)$ are pipe dreams such that $P_\dd\lessdot_{\C}P_{\dd\dd}$. Let \[{X=\{P\in\PD(w):P\wedge P_{\dd\dd}=P_{\dd}\}}.\] Suppose $P_0,P_1,P_2\in X$ are such that $P_0\lessdot_{\C}P_1$ and $P_0\lessdot_{\C}P_2$ (so $P_0=P_1\wedge P_2$). Then there are distinct inversions $(i_1,j_1)$ and $(i_2,j_2)$ such that $P_1=\C_{i_1,j_1}(P_0)$ and $P_2=\C_{i_2,j_2}(P_0)$. Let $R_1=\RR(P_0,P_1)$ and $R_2=\RR(P_0,P_2)$. One may readily check that at least one of the following two statements must hold: 
\begin{itemize}
\item We have $\SW(R_1)\neq\SE(R_2)$, $\SW(R_2)\neq\SE(R_1)$, and $i_1\neq i_2$. 
\item We have $\SW(R_1)\neq\NW(R_2)$, $\SW(R_2)\neq\NW(R_1)$, and $j_1\neq j_2$. 
\end{itemize} 
In either case, we may apply \cref{lem:semi1} or \cref{lem:semi2} to find that $P_1\vee P_2\in X$. 

Observe that $X$ is an order-convex subset of $\PD(w)$ whose minimum element is $P_\dd$. According to \cref{prop:maximum}, the set $X$ has a maximum element. Since the cover relation $P_\dd\lessdot_{\C}P_{\dd\dd}$ was arbitrary, it follows from \cref{prop:Free} that $\PD(w)$ is meet-semidistributive. 
\end{proof}

\section*{Acknowledgments}
Colin Defant was supported by the National Science Foundation under Award No.\ 2201907 and by a Benjamin Peirce Fellowship at Harvard University. Katherine Tung was supported by the Harvard College Research Program. This paper was submitted to arXiv.org on Yellow Pig Day, July 17th, in honor of David C.\ Kelly.

\end{document}